\documentclass[10pt]{article}
\usepackage[a4paper,centering,vscale=0.75,hscale=0.7]{geometry}

\usepackage{mathtools}
\usepackage{amsthm,amssymb}
\usepackage{mathrsfs}

\newtheorem{thm}{Theorem}[section]
\newtheorem{lemma}[thm]{Lemma}
\newtheorem{prop}[thm]{Proposition}

\theoremstyle{remark}
\newtheorem{rmk}[thm]{Remark}

\newcommand{\dom}{\mathsf{D}}

\newcommand{\cB}{\mathscr{B}}
\newcommand{\E}{\mathop{{}\mathbb{E}}}
\newcommand{\cE}{\mathscr{E}}
\newcommand{\cF}{\mathscr{F}}
\newcommand{\cL}{\mathscr{L}}
\newcommand{\cM}{\mathscr{M}}
\renewcommand{\P}{\mathbb{P}}
\newcommand{\erre}{\mathbb{R}}
\newcommand{\enne}{\mathbb{N}}
\newcommand{\eps}{\varepsilon}

\newcommand{\embed}{\hookrightarrow}
\newcommand{\lip}{\dot{C}^{0,1}}

\DeclarePairedDelimiter\abs{\lvert}{\rvert}
\DeclarePairedDelimiter\norm{\lVert}{\rVert}
\DeclarePairedDelimiterX\ip[2]{\langle}{\rangle}{#1,#2}

\numberwithin{equation}{section}

\newif\ifbozza
\bozzafalse

\title{Ergodicity and Kolmogorov equations for dissipative SPDEs with
  singular drift: a variational approach}
  \author{Carlo Marinelli\thanks{Department of Mathematics, University
      College London, Gower Street, London WC1E 6BT, United
      Kingdom. URL: \texttt{http://goo.gl/4GKJP}} 
  \and Luca Scarpa\thanks{Department of Mathematics, University
      College London, Gower Street, London WC1E 6BT, United
      Kingdom. E-mail: \texttt{luca.scarpa.15@ucl.ac.uk}}}
\date{October 16, 2017}

\begin{document}
\maketitle

\begin{abstract}
  We prove existence of invariant measures for the Markovian semigroup
  generated by the solution to a parabolic semilinear stochastic PDE
  whose nonlinear drift term satisfies only a kind of symmetry
  condition on its behavior at infinity, but no restriction on its
  growth rate is imposed. Thanks to strong integrability properties of
  invariant measures $\mu$, solvability of the associated Kolmogorov
  equation in $L^1(\mu)$ is then established, and the infinitesimal
  generator of the transition semigroup is identified as the closure
  of the Kolmogorov operator. A key role is played by a generalized
  variational setting.
\end{abstract}

%%%%%%%%%%%%%%%%%%%%%%%%%%%%%%%%%%%%%%%%%%%%%%%%%%%%%%%%%%%%%%

\section{Introduction}
\label{sec:intro}
Our goal is to study the asymptotic behavior of solutions to
semilinear stochastic partial differential equations on a smooth
bounded domain $D\subseteq\erre^n$ of the form
\begin{equation}
\label{eq:0}
  dX_t + AX_t\,dt + \beta(X_t)\,dt \ni B(X_t)\,dW_t, \qquad X(0)=X_0.
\end{equation}
Here $A:V \to V'$ is a linear maximal monotone operator from a Hilbert
space $V$ to its dual $V'$, and $V \subset H:=L^2(D) \subset V'$ is a
so-called Gelfand triple; $\beta$ is a maximal monotone graph
everywhere defined on $\erre$; $W$ is a cylindrical Wiener process on
a separable Hilbert space $U$, and $B$ takes values in the space of
Hilbert-Schmidt operators from $U$ to $L^2(D)$. Precise assumptions on
the data of the problem are given in {\S}\ref{sec:ass} below. The most
salient point is that $\beta$ is \emph{not} assumed to satisfy any
growth assumption, but just a kind of symmetry on its rate of growth
at plus and minus infinity -- see assumption (vi) in \S\ref{sec:ass}
below. Well-posedness of equation \eqref{eq:0} in the strong
(variational) sense has recently been obtained in \cite{cm:luca} by a
combination of classical results by Pardoux and Krylov-Rozovski{\u\i}
(see~\cite{KR-spde,Pard}) with pathwise estimates and weak compactness
arguments.
The minimal assumptions on the drift term $\beta$ imply that, in
general, the operator $A+\beta$ does not satisfy the coercivity and
boundedness assumptions required by the variational approach of
\cite{KR-spde,Pard}. For this reason, questions such as ergodicity and
existence of invariant measures for \eqref{eq:0} cannot be addressed
using the results by Barbu and Da Prato in \cite{BDP-erg}, which
appear to be the only ones available for equations in the variational
setting (cf. also \cite{cm:DPDE10}).
On the other hand, there is a very vast literature on these problems
for equations cast in the mild setting, references to which can be
found, for instance, in~\cite{DP-K,DZ96,Stannat:rev}. Even in this
case, however, we are not aware of results on equations with a drift
term as general as in \eqref{eq:0}. Our results thus considerably
extend, or at least complement, those on reaction-diffusion equations
in \cite{cerrai-libro,DP-K,DZ96}, for instance, where polynomial
growth assumptions are essential. More recent existence and
integrability results for invariant measures of semilinear equations
have been obtained, e.g., in~\cite{EsS:locL,EsS:impr}, but still under
local Lipschitz-continuity or other suitable growth assumptions on the
drift.
Another possible advantage of our results is that we use only standard
monotonicity assumptions, whereas in a large part of the cited
literature one encounters assumptions of the type
\[
\ip{Ax + \beta(x+y)}{z} \leq f(\norm{y}) - k\norm{x}
\]
for some (or all) $z$ belonging to the subdifferential of $\norm{x}$,
where $f$ is a function and $k$ a constant. Here $A$ actually stands
for the part of $A$ in a Banach space $E$ continuously embedded in
$L^2(D)$, $\ip{\cdot}{\cdot}$ stands for the duality between $E$ and
its dual, and the condition is assumed to hold for those $x$, $y$ for
which all terms are well defined. Often $E$ is chosen as a space of
continuous functions such as $C(\overline{D})$. This monotonicity-type
condition on $A$ and $\beta$ is precisely what one needs in order to
obtain a priori estimates by reducing the original equation to a
deterministic one with random coefficients, under the assumption of
additive noise. Using a figurative but rather accurate expression,
this methods amounts to ``subtracting the stochastic
convolution''. Our estimates are obtained mostly by stochastic
calculus, for which the standard notion of monotonicity suffices.
Among such estimates we obtain the integrability of (the potential of)
the nonlinear drift term $\beta$ with respect to the invariant measure
$\mu$, which is known to be a delicate issue, especially for
non-gradient systems (cf. the discussion in \cite{EsS:locL}). These
results allow us to show that the Kolmogorov operator associated to
the stochastic equation \eqref{eq:0} with additive noise is
essentially $m$-dissipative in $L^1(H,\mu)$. This implies that the
closure of the Kolmogorov operator in $L^1(H,\mu)$ generates a
Markovian semigroup of contractions, which is a $\mu$-version of the
transition semigroup generated by the solution to the stochastic
equation. It is worth mentioning that the variational-type setting,
while allowing for a very general drift term $\beta$, gives raise to
quite many technical issues in the study of Kolmogorov equations, for
instance because test functions in function spaces on $V$ and $V'$
naturally appear.

\smallskip

We conclude this introductory section with a brief description of the
structure of the paper and of the main results.
In Section~\ref{sec:ass} we state the basic assumptions which are in
force throughout the paper, and recall the well-posedness result for
equation \eqref{eq:0} obtained in \cite{cm:luca}.
For the reader's convenience we collect in Section~\ref{sec:prelim}
some tools needed in the sequel, such as Prokhorov's theorem on
compactness of sets of probability measures, and the
Krylov--Bogoliubov criterion for the existence of an invariant
measure for a Markovian transition semigroup.
Section~\ref{sec:aux} is devoted to auxiliary results, most of which
should be interesting in their own right, that underpin our subsequent
arguments. In particular, we prove two generalized versions of the
classical It\^o formula in the variational setting for equation
\eqref{eq:0}: one for the square of the norm, and another one
extending a very useful but not-so-well known version for more general
smooth functions, originally obtained by Pardoux
(see~\cite[p.~62--ff]{Pard}). Furthermore, we establish results on the
first and second-order differentiability, both in the G\^ateaux and
Fr\'echet sense, of (variational) solutions to semilinear equations
with regular drift with respect to the initial datum.
In Section~\ref{sec:inv} we prove that the transition semigroup $P$
generated by the solution to \eqref{eq:0} admits an ergodic invariant
measure $\mu$, which in also shown to be unique and strongly mixing if
$\beta$ is superlinear. These results follow mainly by a priori
estimates (which, in turn, are obtained by stochastic calculus) and
compactness.
Finally, Section~\ref{sec:kolm} deals with the Kolmogorov equation
associated to \eqref{eq:0}. In particular, we characterize the
infinitesimal generator $-L$ of the transition semigroup $P$ on
$L^1(H,\mu)$ as the closure of the Kolmogorov operator $-L_0$. After
showing that $L_0$ is dissipative and coincides with $L$ on a suitably
chosen dense subset of $L^1(H,\mu)$, we prove that the image of
$I+L_0$ is dense in $L^1(H,\mu)$, so that the Lumer-Phillips theorem
can be applied. Due to the variational formulation of the problem, the
latter point turns out to be rather delicate, even though the general
approach follows a typical scheme: we first introduce appropriate
regularizations of $L_0$, for which the Kolmogorov equation can be
solved by established techniques, then we pass to the limit in the
regularization's parameters. Here the generalized It\^o formulas and
the differentiability results proved in Section~\ref{sec:aux} play a
key role.

%--------------------------------

\ifbozza\newpage\else\fi
\section{General assumptions and well-posedness}
\label{sec:ass}
Before stating the hypotheses on the coefficients and on the initial
datum of equation \eqref{eq:0} that will be in force throughout the
paper, let us fix some notation.

\subsection{Notation}
Given two Banach (real) spaces $E$ and $F$, the space of bounded
linear operators from $E$ to $F$ will be denoted by $\cL(E,F)$. When
$F=\erre$, we shall just write $E'$. If $E$ and $F$ are Hilbert
spaces, $\cL^2(E,F)$ stands for the ideal of $\cL(E,F)$ of
Hilbert-Schmidt operators.
The Hilbert space $L^2(D)$ will be denoted by $H$, and its norm and
scalar product by $\norm{\cdot}$ and $\ip{\cdot}{\cdot}$,
respectively.
For any topological space $E$, the Borel $\sigma$-algebra on $E$ will
be denoted by $\cB(E)$. All measures on $E$ are intended to be defined
on its Borel $\sigma$-algebra, unless otherwise stated. The spaces of
bounded Borel-measurable and bounded continuous functions on $E$ will
be denoted by $B_b(E)$ and $C_b(E)$, respectively.

\subsection{Assumptions}
Let $V$ be a separable Hilbert space densely, continuously and
compactly embedded in $H=L^2(D)$. The duality form between $V$ and
$V'$ is also denoted by $\ip{\cdot}{\cdot}$, as customary. We assume
that $A \in \cL(V,V')$ satisfies the following properties:
\begin{itemize}
\item[(i)] there exists $C>0$ such that $\ip{Av}{v}\geq C\norm{v}_V^2$
  for every $v\in V$;
\item[(ii)] the part of $A$ in $H$ can be uniquely extended to an
  $m$-accretive operator $A_1$ on $L^1(D)$;
\item[(iii)] for every $\delta>0$, the resolvent $(I+\delta A_1)^{-1}$
  is sub-Markovian, i.e. for every $f\in L^1(D)$ such that
  $0 \leq f \leq1$ a.e. on $D$, we have
  $0 \leq (I+\delta A)^{-1}f \leq 1$ a.e. on $D$;
\item[(iv)] there exists $m\in\enne$ such that
  $(I+\delta A_1)^{-m} \in \cL(L^1(D), L^\infty(D))$.
\end{itemize}

\medskip

Let us now consider the non-linear term in the drift.
We assume that 
\begin{itemize}
\item[(v)] $\beta \subset \erre \times \erre$ is a maximal monotone
  graph such that $0\in\beta(0)$ and $\dom(\beta)=\erre$.
\end{itemize}
Let $j: \erre \to \erre_+$ be the unique convex lower semicontinuous
function such that $j(0)=0$ and $\beta=\partial j$, in the sense of
convex analysis. We assume that
\begin{itemize}
\item[(vi)] $\displaystyle \limsup_{|r|\to\infty} \frac{j(r)}{j(-r)} <
  \infty.$
\end{itemize}
This hypothesis is obviously satisfied if $j$ (or, equivalently,
$\beta$) is symmetric.
Denoting the convex conjugate of $j$ by $j^*$, it is well known that
the hypothesis $\dom(\beta)=\erre$ is equivalent to the superlinearity
of $j^*$ at infinity, i.e.
\[
  \lim_{|r|\to\infty} \frac{j^*(r)}{|r|} = \infty.
\]
We are going to need the following property implied by assumption
(vi): there exists a strictly positive number $\eta$ such that, for
every measurable function $y:D \to \erre$, $j^*(y) \in L^1(D)$ implies
$j^*(\eta|y|)\in L^1(D)$. In fact, from (vi) we deduce that there
exist $R>0$ and $M_1=M_1(R)>0$ such that $j(r) \leq M_1 j(-r)$ for
$|r|\geq R$. Since $j \geq 0$, one can choose $M_1>1$ without loss of
generality. Setting $M_2:=\max\{j(r): |r|\leq R\}$, which is finite by
continuity of $j$, we deduce that
\[
  j(r) \leq M_1 j(-r) + M_2 \qquad\forall r \in \erre.
\]
Taking convex conjugates on both sides we infer that
\[
  j^*(r) \geq M_1j^*(-r/M_1) - M_2 \qquad\forall r \in \dom(j^*).
\]
Setting $\eta:=1/M_1<1$ and recalling that $j^*(0)=0$, hence $j^*$ is
positive on $\erre$ and increasing on $\erre_+$, one has
\begin{align*}
j^*(\eta\abs{y}) 
&= j^*(\eta y) 1_{\{y \geq 0\}} 
+ j^*(-\eta y) 1_{\{y<0\}}\\
&\leq j^*(y) 1_{\{y \geq 0\}} 
+ \eta j^*(y) 1_{\{y \geq 0\}} + \eta M_2\\
&\leq j^*(y) + M_2 \in L^1(D).
\end{align*}

\medskip

The assumptions on the Wiener process $W$ and the diffusion
coefficient $B$ are standard: let $U$ be a separable Hilbert space and
$W$ a cylindrical Wiener process on $U$, defined on a filtered
probability space $(\Omega,\cF,(\cF_t)_{t \in [0,T]},\P)$ satisfying
the so-called usual conditions.\footnote{Expressions involving random
  elements are always meant to hold $\P$-a.s. unless otherwise
  stated.} We assume that
\begin{itemize}
\item[(vii)] $B:H \to \cL^2(U,H)$ is Lipschitz-continuous and with
  linear growth, i.e. that there exists
a positive constants $L_B$ such that
\begin{align*}
  \norm{B(x)-B(y)}_{\cL^2(U,H)} &\leq L_B\norm{x-y} &\forall x,y\in H,\\
  \norm{B(x)}_{\cL^2(U,H)} &\leq L_B(1+\norm{x}) &\forall x\in H.
\end{align*}
\end{itemize}
Finally, the initial datum $X_0$ is assumed to be $\cF_0$-measurable
and such that $\E\norm{X_0}^2$ is finite.
All hypotheses just stated will be tacitly assumed to hold throughout.

\medskip

The following well-posedness result for equation \eqref{eq:0} has been
proved in \cite{cm:luca}, allowing the diffusion coefficient $B$ to be
also random and time-dependent.
\begin{thm}
  \label{thm:WP}
  There is a unique pair $(X,\xi)$, with $X$ a $V$-valued adapted
  process and $\xi$ an $L^1(D)$-valued predictable process, such
  that
  \begin{gather*}
    X \in L^2(\Omega; L^\infty(0,T; H))\cap L^2(\Omega; L^2(0,T; V)),
    \qquad
    \xi \in L^1(\Omega\times(0,T)\times D),\\
    j(X)+j^*(\xi) \in L^1(\Omega\times(0,T)\times D), \qquad \xi \in \beta(X)
    \quad\text{a.e.~in } \Omega\times(0,T)\times D,
  \end{gather*}
  and
  \[
  X(t) + \int_0^tAX(s)\,ds + \int_0^t\xi(s)\,ds =
  X_0+\int_0^tB(X(s))\,dW(s) \quad\forall t\in[0,T],
  \quad\P\text{-a.s.}
  \]
  in $V' \cap L^1(D)$. Moreover, $X$ is $\P$-a.s. pathwise weakly
  continuous from $[0,T]$ to $H$, and the solution map
  \begin{align*}
    L^2(\Omega;H) &\longrightarrow L^2(\Omega; L^\infty(0,T;H)) \cap
    L^2(\Omega; L^2(0,T; V))\\
    X_0 &\longmapsto X
  \end{align*}
  is Lipschitz-continuous.
\end{thm}
\begin{rmk}
  In a forthcoming work we shall prove that the solution $X$ is
  actually pathwise continuous, not just weakly continuous, and that
  well-posedness continues to hold under local Lipschitz-continuity
  and linear growth assumptions on $B$. We shall also show how the
  regularity of the solution $X$ depends on the regularity of the
  diffusion coefficient $B$.
\end{rmk}

%--------------------------------------------------------------------------

\ifbozza\newpage\else\fi
\section{Preliminaries}
\label{sec:prelim}
\subsection{Compactness in spaces of probability measures}
The set of probability measures on $E$ is denoted by $\cM_1(E)$ and
endowed with the topology $\sigma(\cM_1(E),C_b(E))$, which we shall
call the narrow topology.  We recall that a subset $\mathscr{N}$ of
$\cM_1(E)$ is called (uniformly) \emph{tight} if for every
$\varepsilon>0$ there exists a compact set $K_\varepsilon$ such that
$\mu(E \setminus K_\varepsilon) < \varepsilon$ for all
$\mu \in \mathscr{N}$. The following characterization of relative
compactness of sets of probability measures is classical (see,
e.g.,~\cite[\S 5.5]{Bbk:INT9}).
\begin{thm}[Prokhorov]
  Let $E$ be a complete separable metric space. A subset of $\cM_1(E)$
  is relatively compact in the narrow topology if and only if it is
  tight.
\end{thm}

\subsection{Markovian semigroups and ergodicity}
A family $P=(P_t)_{t\geq 0}$ of Markovian kernels on a measure space
$(E,\cE)$ such that $P_{t+s}=P_tP_s$ for all $t,s \geq 0$ is called a
Markovian semigroup. We recall that a Markovian kernel on $(E,\cE)$ is
a map $K:E \times \cE \to [0,1]$ such that (i) $x \mapsto K(x,A)$ is
$\cE$-measurable for each $A \in \cE$, (ii) $A \mapsto K(x,A)$ is a
measure on $\cE$ for each $x \in E$, and (iii) $K(x,E)=1$ for each
$x \in E$. A Markovian kernel $K$ on $(E,\cE)$ can naturally be
extended to the space $b\cE$ of $\cE$-measurable bounded functions by
the prescription
\[
f \longmapsto Kf := \int_E f(y)\,K(\cdot,dy).
\]
Then $K: b\cE \to b\cE$ is a linear, bounded, positive, $\sigma$-order
continuous map. Similarly, $K$ can be extended to positive measures on
$\cE$ setting
\[
  \mu \longmapsto \mu K(\cdot) := \int_E K(x,\cdot)\,\mu(dx).
\]
The notations $P_tf$ and $\mu P_t$, with $f$ $\cE$-measurable bounded
or positive function and $\mu$ positive measure on $\cE$, are hence to
be understood in this sense. We shall also assume that $P_0=I$ and
that $(t,x) \mapsto P_tf(x)$ is $\cB(\erre_+) \otimes \cE$-measurable.

A probability measure $\mu$ on $\cE$ is said to be an invariant
measure for the Markovian semigroup $P$ if
\[
  \int_E P_t f\,d\mu = \int_E f\,d\mu \qquad \forall f \in b\cE,
  \quad \forall t \geq 0,
\]
or, equivalently, if $\mu P_t = \mu$ for all $t \geq 0$. If $P$
admits an invariant measure $\mu$, then it can be extended to a
Markovian semigroup on $L^p(E,\mu)$, for every $p \geq 1$. 
The invariant measure $\mu$ is said to be ergodic for $P$ if
\[
\lim_{t \to \infty} \frac{1}{t} \int_0^t P_sf \,ds = \int_E f\,d\mu
\qquad \text{ in } L^2(E,\mu) \quad \forall f \in L^2(E,\mu),
\]
and strongly mixing if
\[
\lim_{t\to+\infty} P_tf = \int_E f\,d\mu
\qquad \text{ in } L^2(E,\mu) \quad \forall f \in L^2(E,\mu).
\]
We recall the following classical fact on the structure of the set of
ergodic measures: the ergodic invariant measures for $P$ are the
extremal points of the set of its invariant measures. In particular,
if $P$ admits a unique invariant measure $\mu$, then $\mu$ is
ergodic.
In order to state a criterion for the existence of invariant measures,
let us introduce, for any probability measure $\nu \in \cM_1(E)$, the
family of averaged measures $(\mu^\nu_t)_{t \geq 0}$ defined as
\[
\mu^\nu_t := \frac1t \int_0^t \nu P_s\,ds.
\]
\begin{thm}[Krylov and Bogoliubov]
  Let $(P_t)_{t \geq 0}$ be a (time-homogeneous) Markovian transition
  semigroup on a complete separable metric space $E$. Assume that
  \begin{itemize}
  \item[(a)] $(P_t)_{t \geq 0}$ has the Feller property, i.e. that it
    maps $C_b(E)$ into $C_b(E)$;
  \item[(b)] there exists $\nu \in \cM_1(E)$ such that the
    $(\mu^\nu_t)_{t \geq 0} \subset \cM_1(E)$ is tight.
  \end{itemize}
  Then the set of invariant measures for $(P_t)_{t \geq 0}$ is non-empty.
\end{thm}
Note that if $x \in E$ and $\nu$ is the Dirac measure at $x$, then
$\nu P_s = P_s(x,\cdot)$. Then condition (b) is satisfied if there
exists $x \in E$ such that the family of measures
\[
\biggl( \frac1t \int_0^t P_s(x,\cdot)\,ds \biggr)_{t\geq 0}
\]
is tight. It is easily seen that this latter condition is in turn
satisfied if $(P_t(x,\cdot))_{t \geq 0} \subset \cM_1(E)$
is tight.

\ifbozza\newpage\else\fi
\section{Auxiliary results}\label{sec:aux}
To prove the main results we shall need some auxiliary results that
are interesting in their own right, and that are collected in this
section. In particular, we recall or prove some It\^o-type formulas
and provide conditions for the differentiability of solutions to
equations in variational form with respect to the initial datum.

\subsection{It\^o formulas}
The following version of It\^o's formula for the square of the norm is
\cite[Proposition~6.2]{cm:luca}.
\begin{prop}
  \label{prop:Ito}
  Assume that an adapted process
  \[
    Y \in L^0(\Omega; L^\infty(0,T; H)) \cap L^0(\Omega; L^2(0,T; V))
  \]
  is such that
  \[
    Y(t) + \int_0^tAY(s)\,ds + \int_0^t g(s)\,ds = Y_0 +
    \int_0^tG(s)\,dW(s)
  \]
  in $L^1(D)$ for all $t \in [0,T]$, where
  $Y_0 \in L^0(\Omega,\cF_0; H)$, $G$ is a progressive
  $\cL^2(U,H)$-valued process such that
  \[
  G \in L^2(\Omega\times(0,T);\cL^2(U,H)),
  \]
  $g$ is an adapted $L^1(D)$-valued process such that
  \[
  g \in L^0(\Omega; L^1(0,T; L^1(D))),
  \]
  and there exists $\alpha>0$ for which
  \[
  j(\alpha Y) + j^*(\alpha g) \in L^1(\Omega \times(0,T) \times D).
  \]
  Then  
  \begin{align*}
    &\frac12 \norm{Y(t)}^2 + \int_0^t\ip[\big]{AY(s)}{Y(s)}\,ds
      + \int_0^t\!\!\int_D g(s,x) Y(s,x)\,dx\,ds \\
    &\hspace{3em} = \frac12 \norm{Y_0}^2
      + \frac12\int_0^t \norm{G(s)}^2_{\cL^2(U,H)}\,ds
      + \int_0^t Y(s)G(s)\,dW(s) \qquad\forall t\in[0,T].
  \end{align*}
\end{prop}
\begin{proof}
  Since the resolvent of $A_1$ is ultracontractive by assumption,
  there exists $m \in \enne$ such that
  \[  
  (I+\delta A_1)^{-m}: L^1(D) \to H \qquad \forall \delta>0.
  \]
  Using a superscript $\delta$ to denote the action of
  $(I+\delta A_1)^{-k}$, we have
  \[
    Y^\delta(t) + \int_0^tAY^\delta(s)\,ds + \int_0^t g^\delta(s)\,ds
    = Y^\delta_0 + \int_0^tG^\delta(s)\,dW(s) \qquad\forall t\in[0,T],
  \]
  where $g^\delta \in L^1(0,T; H)$, hence the classical It\^o
  formula yields, for every $\delta>0$,
  \begin{align*}
    &\frac12 \norm{Y^\delta(t)}^2 + \int_0^t\ip[\big]{AY^\delta(s)}{Y^\delta(s)}\,ds
      + \int_0^t\!\!\int_D g^\delta(s,x) Y^\delta(s,x)\,dx\,ds \\
    &\hspace{3em} = \frac12 \norm{Y^\delta_0}^2
      + \frac12\int_0^t \norm{G^\delta(s)}^2_{\cL^2(U,H)}\,ds
      + \int_0^t Y^\delta(s)G^\delta(s)\,dW(s) \qquad\forall t\in[0,T].
  \end{align*}
  We are now going to pass to the limit as $\delta \to 0$. By the
  assumptions on $A$ and the regularity properties of $Y$, $g$, $Y_0$,
  and $G$, one has
  \begin{align*}
  Y^\delta(t) \to Y(t) \qquad&\text{in } H \quad\forall t\in[0,T],\\
  Y^\delta \to Y \qquad&\text{in } L^2(0,T; V),\\
  AY^\delta \to AY \qquad&\text{in } L^2(0,T; V'),\\
  g^\delta\to g \qquad&\text{in } L^1(0,T; L^1(D)),\\
  Y_0^\delta \to Y_0 \qquad&\text{in } H,\\
  G^\delta\to G \qquad&\text{in } L^2(\Omega; L^2(0,T; \cL^2(U,H))).
  \end{align*}
  This implies
  \[
  \int_0^t\ip[\big]{AY^\delta(s)}{Y^\delta(s)}\,ds \longrightarrow
  \int_0^t\ip[\big]{AY(s)}{Y(s)}\,ds
  \]
  and
  \[
  \int_0^t \norm{G^\delta(s)}^2_{\cL^2(U,H)}\,ds \longrightarrow 
  \int_0^t \norm{G(s)}^2_{\cL^2(U,H)}\,ds.
  \]
  Using the dominated convergence theorem, it is
  not difficult to check that $\norm{Y^\delta G^\delta -
    YG}^2_{\cL^2(U,\erre)}$ converges to zero in probability, hence
  also (along a subsequence)
  \[
  \int_0^t Y^\delta(s)G^\delta(s)\,dW(s) \longrightarrow
  \int_0^t Y(s)G(s)\,dW(s).
  \]
  Finally, the symmetry assumption on $j$ ensures that
  $(g^\delta Y^\delta)$ is uniformly integrable on $(0,T)\times D$, so
  that
  \[
    \int_0^t\!\!\int_D g^\delta(s,x) Y^\delta(s,x)\,dx\,ds \to
    \int_0^t\!\!\int_D g(s,x) Y(s,x)\,dx\,ds.
    \qedhere
  \]
\end{proof}

We shall also need a simplified version of an It\^o formula in the
variational setting, due to Pardoux, for functions more general than
the square of the $H$-norm. For its proof (in a more general context)
we refer to \cite[p.~62-ff.]{Pard}.
\begin{prop}
  \label{prop:Ito_pard}
  Let $Y \in L^0(\Omega;L^2(0,T;V))$ be such that 
  \[
  Y(t)=Y_0 + \int_0^tv(s)\,ds + \int_0^tG(s)\,dW(s)
  \]
  for all $t\in[0,T]$, where $Y_0 \in L^0(\Omega, \cF_0, \P; H)$ and
  \[
  v \in L^0(\Omega; L^1(0,T; H)) \oplus L^0(\Omega; L^2(0,T; V'))
  \]
  is adapted and $G\in L^2(\Omega\times(0,T);\cL^2(U,H))$
  is progressively measurable. Then, for any $F \in C^2_b(H) \cap C^1_b(V)$, one has
  \begin{align*}
    F(Y(t))
    &= F(Y_0) + \int_0^t DF(Y(s))v(s) \,ds
      + \int_0^t DF(Y(s)) G(s)\,dW(s)\\
    & \quad + \frac12 \int_0^t \operatorname{Tr}\bigl(%
      G^*(s)D^2F(Y(s))G(s)\bigr)\,ds
  \end{align*}
  for every $t \in [0,T]$, $\P$-almost surely.
\end{prop}

The previous It\^o formula can be extended to processes satisfying
weaker integrability conditions, in analogy to Proposition \ref{prop:Ito}.
\begin{prop}\label{prop:Ito_pard2}
  Let $Y \in L^0(\Omega;L^2(0,T;V))\cap L^0(\Omega; L^\infty(0,T; H))$
  be such that
  \[
  Y(t)=Y_0 + \int_0^tAv(s)\,ds + \int_0^t g(s)\,ds +  \int_0^tG(s)\,dW(s)
  \]
  for all $t\in[0,T]$, where $Y_0 \in L^0(\Omega, \cF_0, \P; H)$ and
  \[
  v \in L^0(\Omega; L^2(0,T; V)), \qquad g \in L^0(\Omega; L^1(0,T; L^1(D)))
  \]
  are adapted and $G\in L^2(\Omega\times(0,T);\cL^2(U,H))$
  is progressively measurable. Then, for any $F \in C^2_b(H) \cap C^1_b(V') \cap
  C^1_b(L^1(D))$, one has
  \begin{align*}
    F(Y(t))
    &= F(Y_0) + \int_0^t \ip{Av(s)}{DF(Y(s))} \,ds
      + \int_0^t\!\!\int_Dg(s)DF(Y(s))\,ds\\
    & \quad + \int_0^t DF(Y(s)) G(s)\,dW(s)
      + \frac12 \int_0^t \operatorname{Tr}\bigl(%
      G^*(s)D^2F(Y(s))G(s)\bigr)\,ds
  \end{align*}
  for every $t \in [0,T]$, $\P$-a.s..
\end{prop}
\begin{proof}
  Since the resolvent of $A_1$ is ultracontractive by assumption,
  there exists $m \in \enne$ such that
  \[  
  (I+\delta A_1)^{-m}: L^1(D) \to H \qquad\forall \delta>0.
  \]
  Using a superscript $\delta$ to denote the action of $(I+\delta
  A_1)^{-m}$, we have
  \[
  Y^\delta(t) = Y^\delta_0 + \int_0^tAv^\delta(s)\,ds + \int_0^t
  g^\delta(s)\,ds + \int_0^tG^\delta(s)\,dW(s) \qquad\forall
  t\in[0,T],
  \]
  where $Av^\delta+g^\delta \in L^0(\Omega; L^1(0,T; H)) \oplus
  L^0(\Omega; L^2(0,T; V'))$. Hence, by
  Proposition~\ref{prop:Ito_pard}, for every $\delta>0$ we have
 \begin{align*}
   F(Y^\delta(t))
   &= F(Y^\delta_0) + \int_0^t \ip{Av^\delta(s)}{DF(Y^\delta(s))} \,ds + \int_0^t\!\!\int_Dg^\delta(s)DF(Y^\delta(s))\,ds\\
   & \quad + \int_0^t DF(Y^\delta(s)) G^\delta(s)\,dW(s) + \frac12
   \int_0^t \operatorname{Tr}\bigl(%
   (G^\delta)^*(s)D^2F(Y^\delta(s))G^\delta(s)\bigr)\,ds
  \end{align*}
  for every $t \in [0,T]$, $\P$-almost surely. Let us pass to the
  limit as $\delta \to 0$ in the previous equation.  It is clear from
  the fact that $Y(t), Y_0 \in H$ and the continuity of $F$ that
  \[
    F(Y^\delta(t)) \to F(Y(t)), \qquad F(Y_0^\delta) \to F(Y_0).
  \] 
  Moreover, since $v^\delta+\delta Av^\delta=v$ in $V$, taking the
  duality pairing with $Av^\delta \in V'$, we have
  \[
  \ip{Av^\delta}{v^\delta} + \delta\norm{Av^\delta}^2 = \ip{Av^\delta}{v}\leq
  \norm{A}_{\cL(V,V')}\norm{v^\delta}_V\norm{v}_V,
  \]
  from which, by coercivity of $A$,
  \[
  \norm{v^\delta}_V\leq \frac{\norm{A}_{\cL(V,V')}}{C}\norm{v}_V
  \qquad \forall \delta>0.
  \]
  Taking into account that $v \in L^2(0,T; V)$, we deduce that
  $v^\delta \to v$ weakly in $L^2(0,T; V)$. Since $Y^\delta \to Y$ in
  $L^2(0,T; H)$, by continuity of $A$ and the fact that $DF\in
  C_b(H,V)$, we have $Av^\delta \to Av$ weakly in $L^2(0,T; V')$ and
  $DF(Y^\delta)\to DF(Y)$ in $L^2(0,T; V)$, hence
  \[
  \int_0^t \ip{Av^\delta(s)}{DF(Y^\delta(s))} \,ds \longrightarrow
  \int_0^t \ip{Av(s)}{DF(Y(s))} \,ds.
  \]
  Furthermore, since $Y^\delta(t) \to Y(t)$ in $H$ for every
  $t\in[0,T]$, recalling that $DF \in C_b(H,L^\infty(D))$ and
  $g^\delta \to g$ in $L^1(0,T; L^1(D))$, we have (possibly along a
  subsequence)
  \[
  \int_D g^\delta(s)DF(Y^\delta(s)) \longrightarrow \int_D g(s)DF(Y(s))
  \qquad\text{for a.e.~} s \in (0,T).
  \]
  Taking into account that $\int_Dg^\delta DF(Y^\delta)\leq
  \norm{DF}_{C_b(H,L^\infty(D))}\norm{g}_{L^1(D)}\in L^1(0,T)$, by the
  dominated convergence theorem we then have
  \[
  \int_0^t\!\!\int_Dg^\delta(s)DF(Y^\delta(s))\,ds \longrightarrow
  \int_0^t\!\!\int_Dg(s)DF(Y^\delta(s))\,ds.
  \]
  Moreover, since $Y^\delta(t)\to Y(t)$ in $H$ for every $t\in[0,T]$,
  recalling that $D^2F \in C(H,\cL(H))$ and $G^\delta \to G$ in
  $L^2(\Omega;L^2(0,T;\cL^2(U,H)))$, we have (possibly along a
  subsequence)
  \[
  \operatorname{Tr}\bigl(%
      (G^\delta)^*(s)D^2F(Y^\delta(s))G^\delta(s)\bigr) \to
      \operatorname{Tr}\bigl(%
      G^*(s)D^2F(Y(s))G(s)\bigr) \qquad\text{for a.e.~}s\in(0,T).
  \]
  Since $\operatorname{Tr}\bigl(
  (G^\delta)^*D^2F(Y^\delta)G^\delta\bigr)\leq
  \norm{D^2F}_{C(H,\cL(H))}\norm{G}^2_{\cL^2(U,H)}\in L^1(0,T)$,
  the dominated convergence theorem yields
  \[
  \int_0^t \operatorname{Tr}\bigl(%
      (G^\delta)^*(s)D^2F(Y^\delta(s))G^\delta(s)\bigr)\,ds
  \longrightarrow
  \int_0^t \operatorname{Tr}\bigl(%
     G^*(s)D^2F(Y(s))G(s)\bigr)\,ds.
  \]
  Finally, by the Davis inequality and the ideal property of
  Hilbert-Schmidt operators, we have
  \begin{align*}
    &\E\sup_{t\in[0,T]} \biggl| \int_0^t DF(Y^\delta(s)) G^\delta(s)\,dW(s)
      - \int_0^t DF(Y(s)) G(s)\,dW(s) \biggr|\\
    &\quad \lesssim
      \E\left(\int_0^T\norm[\big]{DF(Y^\delta(s)) G^\delta(s) -
      DF(Y(s)) G(s)}^2_{\cL(U,\erre)}\,ds\right)^{1/2}\\
    &\quad \lesssim\E\left(\int_0^T\norm{DF(Y^\delta(s))}^2
      \norm{G^\delta(s)-G(s)}^2_{\cL^2(U,H)}\,ds\right)^{1/2}\\
    &\qquad\qquad + \E\left(\int_0^T\norm{DF(Y^\delta(s))-DF(Y(s))}^2
      \norm{G^\delta(s)}_{\cL^2(U,H)}^2\,ds\right)^{1/2}\\
    &\quad \leq\norm{DF}_{C(H,H)}
      \norm{G^\delta-G}_{L^2(\Omega;L^2(0,T;\cL^2(U,H)))}\\
    &\qquad\qquad + \E\left(\int_0^T\norm{G(s)}^2_{\cL^2(U,H)}
      \norm{DF(Y^\delta(s))-DF(Y(s))}^2 \,ds\right)^{1/2},
  \end{align*}
  where the first term on the right-hand side converges to $0$ because
  \[
    G^\delta \to G \qquad \text{in } L^2(\Omega; L^2(0,T; \cL^2(U,H))).
  \]
  Similarly,
  since $DF(Y^\delta)\to DF(Y)$ a.e., it follows by the dominated
  convergence theorem that the second term on the right-hand side
  converges to zero as well. Therefore, passing to subsequence if
  necessary, one has
  \[
  \int_0^t DF(Y^\delta(s)) G^\delta(s)\,dW(s) \longrightarrow
  \int_0^t DF(Y(s)) G(s)\,dW(s).
  \qedhere
  \]
\end{proof}

\subsection{Differentiability with respect to the initial datum for
  solutions to equations in variational form}
\label{ssec:diff}
Let $g \in C^2_b(\erre)$ and consider the equation
\[
dX + AX\,dt = g(X)\,dt + G\,dW, \qquad X(0)=x,
\]
in the variational sense, where $A$ satisfies the hypotheses of
Section~\ref{sec:ass}, $G \in \cL^2(U,H)$, and $x \in H$.

For compactness of notation we shall write $E$ in place of
$C([0,T];H) \cap L^2(0,T;V)$. The above equation admits a unique
variational solution $X^x \in L^2(\Omega;E)$. Here and in the
following we often use superscripts to denote the dependence on the
initial datum. We are going to provide sufficient conditions ensuring
that the solution map $x \mapsto X^x$ belongs to
$C^2_b(H;L^2(\Omega;E))$. The problem of regular dependence on the
initial datum for equations in the variational setting does not seem
to be addressed in the literature. On the other hand, several results
are available for mild solutions (see, e.g.,
\cite{cerrai-libro,DZ96,cm:JFA10}), where an approach via the implicit
function theorem depending on a parameter is adopted. Here we proceed
in a more direct and, we believe, clearer way. The results are
non-trivial (and probably not easily accessible via the implicit
function theorem) in the sense that the solution map is Fr\'echet
differentiable even though, as is well known, the superposition
operator associated to $g$ is \emph{never} Fr\'echet differentiable
unless $g$ is affine.
The first and second Fr\'echet derivative of the solution map shall be
denoted by $DX$ and $D^2X$, respectively. These are maps with domain
$H$ and codomain $\cL(H,L^2(\Omega;E))$ and $\cL_2(H;L^2(\Omega;E))$,
respectively. Here and in the following we denote the space of
continuous bilinear mappings from $H \times H$ to a Banach space $F$
by $\cL_2(H;F)$.

We begin with first-order differentiability.
\begin{thm}   \label{thm:d1}
  The solution map $x \mapsto X^x: H \to L^2(\Omega;E)$ is
  continuously (Fr\'echet) differentiable with bounded derivative.
  Moreover, for any $h \in H$, setting $Y_h := (DX)h$, one has
  \begin{equation}   \label{eq:dg}
  Y_h' + AY_h = g'(X^x)Y_h, \qquad Y_h(0)=h,
  \end{equation}
  in the variational sense.
\end{thm}
\begin{proof}
  Classical (deterministic) results imply that \eqref{eq:dg} admits a
  unique solution $Y_h \in E$ for $\P$-a.e. $\omega \in \Omega$. Since
  $X^x$ is an adapted process and $h$ is non-random, it follows that
  $Y_h$ is itself adapted. Alternatively, and more directly, one can
  apply the stochastic variational theory to \eqref{eq:dg}, deducing
  that $Y_h \in L^2(\Omega;E)$ is adapted.
  Let us set, for compactness of notation,
  \[
    X_\varepsilon := X^{x + \varepsilon h}, \qquad
    z_\varepsilon := \frac{1}{\varepsilon}(X_\varepsilon - X) - Y_h,
  \]
  where $\varepsilon$ is an arbitrary real number. Elementary
  calculations show that
  \[
    z_\varepsilon(t) + \int_0^t Az_\varepsilon(s)\,ds
    = \int_0^t \Bigl( \frac{1}{\varepsilon}
    \bigl( g(X_\varepsilon(s)) - g(X(s)) \bigr) - g'(X(s))Y_h(s) \Bigr)\,ds.
  \]
  Writing
  \[
    g(X_\varepsilon) - g(X) = g(X+\varepsilon Y_h) - g(X)
    + g(X_\varepsilon) - g(X+\varepsilon Y_h)
  \]
  yields
  \begin{align*}
    \frac{1}{\varepsilon} \bigl( g(X_\varepsilon) - g(X) \bigr)
    - g'(X)Y_h
    &= \frac{1}{\varepsilon} \bigl( g(X+\varepsilon Y_h) - g(X) \bigr)
    - g'(X)Y_h\\
    &\quad + \frac{1}{\varepsilon} \bigl(%
      g(X_\varepsilon) - g(X+\varepsilon Y_h) \bigr)\\
    &=: R_\varepsilon + S_\varepsilon.
  \end{align*}
  By the integration-by-parts formula applied to the equation for
  $z_\varepsilon$ we get
  \[
    \frac12 \norm{z_\varepsilon(t)}^2
    + \int_0^t \ip{Az_\varepsilon(s)}{z_\varepsilon(s)}\,ds =
    \int_0^t \ip{R_\varepsilon(s)}{z_\varepsilon(s)}\,ds
    + \int_0^t \ip{S_\varepsilon(s)}{z_\varepsilon(s)}\,ds,
  \]
  where
  $\ip{S_\varepsilon}{z_\varepsilon} \leq \norm{S_\varepsilon}
  \norm{z_\varepsilon}$ and, by the Lipschitz continuity of $g$,
  \[
    \norm{S_\varepsilon} \leq \norm{g}_{\lip} \, \frac{1}{\varepsilon}
    \norm[\big]{X_\varepsilon-X-\varepsilon Y_h}
    = \norm{g}_{\lip} \norm{z_\varepsilon},
  \]
  so that
  $\ip{S_\varepsilon}{z_\varepsilon} \leq \norm{g}_{\lip}
  \norm{z_\varepsilon}^2$. Since
  $\ip{R_\varepsilon}{z_\varepsilon} \leq \bigl(
  \norm{R_\varepsilon}^2 + \norm{z_\varepsilon}^2 \bigr)/2$, we are
  left with
  \[
    \frac12 \norm{z_\varepsilon(t)}^2 + \int_0^t
    \ip{Az_\varepsilon(s)}{z_\varepsilon(s)}\,ds \leq
    \bigl( 1/2 + \norm{g}_{\lip} \bigr) \int_0^t \norm{z_\varepsilon(s)}^2\,ds
    + \frac12 \int_0^t \norm{R_\varepsilon(s)}^2\,ds.
  \]
  For an arbitrary $t > 0$ one has, by the coercivity of $A$,
  \[
    \frac12 \norm{z_\varepsilon}^2_{C([0,t];H)} 
    + C\int_0^t \norm{z_\varepsilon(s)}_V^2\,ds
    \leq \bigl( 1 + 2\norm{g}_{\lip} \bigr)
         \int_0^{t} \norm{z_\varepsilon}^2_{C([0,s];H)} \,ds
    + \int_0^{t} \norm{R_\varepsilon(s)}^2\,ds,
  \]
  hence also, by Fubini's theorem and Gronwall's inequality,
  \[
    \E\norm{z_\varepsilon}^2_{C([0,T];H)} \leq
    e^{(2 + 4\norm{g}_{\lip})T}
    \E\int_0^T \norm{R_\varepsilon(s)}^2\,ds.
  \]
  It is clear from the hypotheses on $g$ and the definition of
  $R_\varepsilon$ that $R_\varepsilon \to 0$ in
  $L^0(\Omega \times [0,T] \times D)$ as $\varepsilon \to 0$ for every
  $s \in [0,T]$. Moreover, it follows by the Lipschitz continuity of
  $g$ and elementary estimates that
  $\abs{R_\varepsilon} \lesssim \norm{g}_{\lip} \abs{Y_h}$, where
  $Y_h \in L^2(\Omega \times [0,T] \times D)$. The dominated convergence
  theorem thus yields
  \[
  \lim_{\varepsilon \to 0} \E\int_0^T \norm{R_\varepsilon(s)}^2\,ds = 0.
  \]
  Since
  \[
    C \E\int_0^T \norm{z_\varepsilon(s)}^2_V\,ds \leq
    \bigl( 1 + 2\norm{g}_{\lip} \bigr) T \E\norm{z_\varepsilon}^2_{C([0,T];H)}
    + \E\int_0^T \norm{R_\varepsilon(s)}^2\,ds,
  \]
  we conclude that
  \[
    \lim_{\varepsilon \to 0}
    \norm[\big]{z_\varepsilon}_{L^2(\Omega;E)} = 0.
  \]
  This proves that the solution map is differentiable in every
  direction of $H$, and that its directional derivative in the
  direction $h \in H$ is given by the (unique) solution $Y_h$ to
  \eqref{eq:dg}. It is then clear that the map $h \mapsto Y_h$ is
  linear. Let us prove that it is also continuous: in analogy to
  computations already carried out above, the integration-by-parts
  formula yields
  \[
    \frac12 \norm{Y_h(t)}^2 + \int_0^t \ip{AY_h(s)}{Y_h(s)}\,ds
    = \norm{h}^2 + \int_0^t \ip{g'(X^x(s))Y_h(s)}{Y_h(s)}\,ds,
  \]
  from which one infers
  \[
    \norm[\big]{Y_h}^2_{C([0,t];H)} + \norm[\big]{Y_h}^2_{L^2(0,t;V)}
    \lesssim \norm{h}^2 + \int_0^t \norm[\big]{Y_h}^2_{C([0,s];H)}\,ds,
  \]
  hence also, by Gronwall's inequality and elementary estimates,
  \[
    \norm[\big]{Y_h}_E \lesssim \norm{h}.
  \]
  It is important to note that this inequality holds $\P$-a.s. with a
  non-random implicit constant that depends only on $T$ and on the
  Lipschitz constant of $g$, but not on the initial datum $x$. From
  this it follows that
  \[
  \norm[\big]{Y_h}_{L^p(\Omega;E)} \lesssim_T \norm{h}
  \qquad \forall p \geq 0,
  \]
  hence, in particular, that $h \mapsto Y_h$ is the G\^ateaux
  derivative of $x \mapsto X^x$.
  Setting $Y^x:=h \mapsto Y_h$, we are going to prove that the map
  \begin{align*}
    H &\longrightarrow \cL(H,L^2(\Omega;E))\\
    x &\longmapsto Y^x
  \end{align*}
  is continuous. This implies, by a well-known criterion (see, e.g.,
  \cite[Theorem~1.9]{AmbPro}), that $x \mapsto X^x$ is Fr\'echet
  differentiable with Fr\'echet derivative (necessarily) equal to
  $Y^x$. Let $(x_n) \subset H$ be a sequence converging to $x$ in $H$,
  and write for simplicity $X^n:=X^{x_n}$, $Y^n:=Y^{x_n}$, $X:=X^x$, and
  $Y:=Y^x$, with a subscript $h$ to denote their action on a fixed
  element $h \in H$.  One has
  \[
    Y_h^n(t) - Y_h(t) + \int_0^t A(Y_h^n(s)-Y_h(s))\,ds
    = x_n - x + \int_0^t \bigl( g'(X^n)Y_h^n - g'(X)Y_h \bigr)(s)\,ds,
  \]
  for which the integration-by-parts formula yields
  \begin{align*}
    &\frac12 \norm[\big]{Y_h^n(t) - Y_h(t)}^2
      + C \int_0^t \norm[\big]{Y_h^n(s) - Y_h(s)}_V^2\,ds\\
    &\hspace{5em} \leq \frac12 \norm{x^n - x}^2
      + \int_0^t \ip[\big]{g'(X^n)Y_h^n - g'(X)Y_h}{Y_h^n-Y_h}(s)\,ds,
  \end{align*}
  where
  \begin{align*}
  \ip[\big]{g'(X^n)Y_h^n - g'(X)Y_h}{Y_h^n-Y_h}
  &= \ip[\big]{g'(X_n)(Y_h^n-Y_h)}{Y_h^n-Y_h}\\
  &\quad + \ip[\big]{(g'(X^n)-g'(X))Y_h}{Y_h^n-Y_h},
  \end{align*}
  so that, by elementary estimates,
  \begin{align*}
    & \norm[\big]{Y_h^n(t) - Y_h(t)}^2
    + 2C \int_0^t \norm[\big]{Y_h^n(s) - Y_h(s)}_V^2\,ds\\
    &\hspace{5em} \leq \norm{x_n - x}^2 
    + \bigl( 2\norm{g}_{\lip} + 1 \bigr)
      \int_0^t \norm[\big]{Y_h^n(s) - Y_h(s)}^2\,ds\\
    &\hspace{5em} \quad + \int_0^t \norm[\big]{%
    \bigl(g'(X^n(s)) - g'(X(s))\bigr)Y_h(s)}^2\,ds.
  \end{align*}
  Taking the supremum in time, Gronwall's inequality implies
  \[
  \norm[\big]{Y^n_h - Y_h}_E \lesssim
  \norm{x_n-x} + \norm[\big]{(g'(X^n)-g'(X))Y_h}_{L^2(0,T;H)},
  \]
  where the implicit constant depends on $C$, $T$ and on the Lipschitz
  constant of $g$. Furthermore, since, as observed above, $h \mapsto
  Y_h$ is a linear bounded map from $H$ to $C([0,T];H)$ $\P$-a.s. with
  non-random operator norm, i.e.
  \[
  \sup_{\norm{h}\leq 1} \norm{Y_h}_{C(0,T];H)}
  \lesssim_{T,g} 1,
  \]
  one has
  \[
  \E \sup_{\norm{h}\leq 1}
  \norm[\big]{(g'(X^n)-g'(X))Y_h}_{L^2(0,T;H)} \lesssim
  \E\norm[\big]{g'(X^n)-g'(X)}_{C(0,T];H)},
  \]
  and the last term converges to zero as $n \to \infty$ by the
  dominated convergence theorem, because $X^n \to X$ in
  $L^2(\Omega;C([0,T];H))$ and $g \in C^2_b$ (in particular, $g'$ is
  Lipschitz-continuous). It immediately follows that $x \mapsto Y^x$
  is a continuous map on $H$ with values in $\cL(H,L^2(\Omega;E))$.
  Furthermore, since we have shown that $\norm{Y^x_h}_{L^p(\Omega;E)}
  \lesssim \norm{h}$ for all $p \geq 0$ with a constant independent of
  $x$, we conclude that $x \mapsto X^x$ is of class $C^1_b$ from $H$
  to $L^2(\Omega;E)$.
\end{proof}

To establish the second-order Fr\'echet differentiability of
$x \mapsto X^x$, it is convenient to consider the equation
\begin{equation}
  \label{eq:d2g}
  Z'_{hk} + AZ_{hk} = g'(X)Z_{hk} + g''(X)Y_hY_k, \qquad Z_{hk}(0)=0,
\end{equation}
where $h,k\in H$ and $Y_h, Y_k$ are the solutions to \eqref{eq:dg}
with initial conditions $h$ and $k$, respectively. This is manifestly
the equation formally satisfied by the second-order Fr\'echet
derivative of $x \mapsto X^x$ evaluated at $(h,k)$.

In order to prove that \eqref{eq:d2g} is well-posed, we need the
following lemma, which is probably well known, but for which we could
not find a reference, except for the classical case where
$f \in L^2(0,T;V')$ (see, e.g., \cite{LiMa1}).
\begin{lemma}   \label{lm:sola}
  Let $y_0\in H$, $f\in L^1(0,T; H)$, and
  $\ell\in L^\infty((0,T)\times D)$. Then there exists a unique
  \[
  y\in C([0,T]; H)\cap L^2(0,T; V)
  \]
  such that
  \[
    y(t) +\int_0^t Ay(s)\,ds = y_0 +\int_0^t\ell(s)y(s)\,ds +
    \int_0^tf(s)\,ds \qquad \forall t \in [0,T].
  \]
  Moreover, one has
  \[
    \frac12\norm{y(t)}^2 + \int_0^t\ip{Ay(s)}{y(s)}\,ds =
    \frac12\norm{y_0}^2 + \int_0^t\!\!\int_D\ell(s)|y(s)|^2\,ds +
    \int_0^t\ip{f(s)}{y(s)}\,ds \qquad\forall t \in [0,T].
  \]
\end{lemma}
\begin{proof}
  Let $(f_n)$ be a sequence in $L^2(0,T;H)$ such that $f_n \to f$ in
  $L^1(0,T; H)$ as $n \to \infty$. By the variational theory of
  deterministic equations, for every $n \in \enne$ there exists a
  unique
  \[
    y_n \in H^1(0,T; V') \cap L^2(0,T; V) \embed C([0,T]; H)
  \]
  such that
  \[
    y'_n(t) + Ay_n(t) = \ell(t)y_n(t) + f_n(t) \quad\text{in $V'$ for
      a.e. } t\in(0,T)\,, \qquad y_n(0)=y_0.
  \]
  Therefore, for every $n$, $m\in\enne$, the integration-by-parts
  formula and an easy computation show that
  \begin{align*}
    &\norm{y_n(t)-y_m(t)}^2 + 2C\int_0^t\norm{y_n(s)-y_m(s)}_V^2\,ds \\
    &\quad \leq 2\norm{\ell}_{L^\infty((0,T)\times D)}
      \int_0^t\norm{y_n(s)-y_m(s)}^2\,ds +
      2\int_0^t\ip{f_n(s)-f_m(s)}{y_n(s)-y_m(s)}\,ds\\
    &\quad \leq 2\norm{\ell}_{L^\infty((0,T)\times D)}
      \int_0^t\norm{y_n(s)-y_m(s)}^2\,ds +
      2\norm{y_n-y_m}_{C([0,t]; H)}\norm{f_n-f_m}_{L^1(0,T; H)}
  \end{align*}
  for every $t \in [0,T]$. By the Young inequality we infer then that,
  for every $\eps\geq0$,
  \begin{align*}
    &\norm{y_n-y_m}^2_{C([0,t]; H)} + \norm{y_n-y_m}_{L^2(0,t; V)}^2\\
    &\quad \lesssim \eps\norm{y_n-y_m}^2_{C([0,t]; H)}
    +\frac1{4\eps}\norm{f_n-f_m}_{L^1(0,T; H)}^2 +
    \int_0^t\norm{y_n-y_m}^2_{C([0,s]; H)}\,ds
  \end{align*}
  for every $t\in[0,T]$, from which, thanks to Gronwall's inequality, 
  \[
    \norm{y_n-y_m}_{C([0,T]; H)\cap L^2(0,T; V)}
    \lesssim \norm{f_n-f_m}_{L^1(0,T; H)}.
  \]
  We deduce that there exists $y \in C([0,T]; H)\cap L^2(0,T; V)$ such
  that
  \[
    y_n \to y \qquad \text{in } C([0,T]; H)\cap L^2(0,T; V).
  \]
  It clear follows from $y \in L^2(0,T; V)$ and
  $A \in \cL(V,V')$ that $Ay\in L^2(0,T; V')$ and $Ay_n \to Ay$
  in $L^2(0,T; V')$ as $n \to \infty$. Moreover, we also have that
  \[
    \frac12\norm{y_n(t)}^2 + \int_0^t\ip{Ay_n(s)}{y_n(s)}\,ds =
    \frac12\norm{y_0}^2 + \int_0^t\!\!\int_D\ell(s)|y_n(s)|^2\,ds+
    \int_0^t\ip{f_n(s)}{y_n(s)}\,ds
  \]
  for all $t \in [0,T]$. Hence the last assertion follows letting
  $n \to\infty$. The uniqueness of $y$ is a consequence of the
  monotonicity of $A$.
\end{proof}

In order to prove second-order Fr\'echet differentiability of the
solution map $x \mapsto X^x$ we need to make the further assumption
that $V$ is continuously embedded in $L^4(D)$. This is satisfied, for
instance, if $V=H^1_0$ and $d \leq 4$. In fact, by the Sobolev
embedding theorem, $H^1_0 \embed L^{2^*}$, where
\[
\frac{1}{2^*} = \frac12 - \frac1d
\]
for $d\geq 3$ and $2^*=+\infty$ otherwise.

We proceed as follows: first we establish well-posedness for equation
\eqref{eq:d2g}, and then we show that its unique solution identifies
$D^2X$.
\begin{prop}
  Assume that $V$ is continuously embedded in $L^4(D)$. Then equation
  \eqref{eq:d2g} admits a unique variational solution $Z_{hk}$ for any
  $h$, $h \in H$. Moreover, the map
  \[
  Z^x: H\times H\to L^2(\Omega, E), \qquad (h,k)\mapsto Z_{hk}^x
  \]
  is bilinear and continuous for any $x\in H$, and there exists a positive constant $M>0$ such that 
  \[
  \norm{Z^x}_{\cL_2(H;L^2(\Omega; E))}\leq M \qquad \forall x\in H.
  \]
\end{prop}
\begin{proof}
  H\"older's inequality and the boundedness of $g''$ yield
  \[
  \norm{g''(X)Y_hY_k}\leq
  \norm{g''}_{\lip} \norm{Y_h}_{L^4(D)}\norm{Y_k}_{L^4(D)}\lesssim
  \norm{Y_h}_V\norm{Y_k}_V,
  \]
  so that $g''(X)Y_hY_k \in L^1(0,T; H)$ since $Y_h, Y_k \in L^2(0,T;
  V)$.  Hence, by Lemma~\ref{lm:sola} there is a unique
  \[
  Z_{hk} \in C([0,T]; H)\cap L^2(0,T; V)
  \]
  such that
  \[
  Z_{hk}(t) + \int_0^tAZ_{hk}(s)\,ds = \int_0^tg'(X(s))Z_{hk}(s)\,ds +
  \int_0^tg''(X(s))Y_h(s)Y_k(s)\,ds \qquad\forall t\in[0,T].
  \]
  Let us show that $(h,k)\mapsto Z_{hk}$ is a continuous bilinear map.
  The bilinearity is clear from equation \eqref{eq:d2g}.  Moreover,
  testing by $Z_{hk}$ and using the coercivity of $A$ we have that
  \begin{align*}
    \norm{Z_{hk}(t)}^2
    &+ \int_0^t\norm{Z_{hk}(s)}_V^2\,ds\lesssim
      \norm{g}_{C^1_b}\int_0^t\norm{Z_{hk}(s)}^2\,ds
      +\norm{g}_{C^2_b}\int_0^t\norm{Y_h(s)}_V\norm{Y_k(s)}_V\,ds\\
    &\leq\norm{g}_{C^1_b}\int_0^t\norm{Z_{hk}(s)}^2\,ds
      + \norm{g}_{C^2_b}\norm{Y_h}_{L^2(0,T; V)}\norm{Y_k}_{L^2(0,T; V)}\\
    &\lesssim_T\norm{g}_{C^1_b}\int_0^t\norm{Z_{hk}(s)}^2\,ds +
      \norm{g}_{C^2_b}\norm{h}\norm{k}
  \end{align*}
  and Gronwall's inequality yields
  \[
  \norm{Z^x_{hk}}_{L^2(\Omega; C([0,T]; H))\cap L^2(\Omega; L^2(0,T;
    V))}\lesssim \norm{h}\norm{k} \qquad\forall h,k,x\in H,
  \]
  from which the last assertion follows.
\end{proof}

\begin{thm}\label{thm:d2}
  Assume that $V$ is continuously embedded in $L^4(D)$. Then the
  solution map $x \mapsto X^x$ is of class $C^2_b$ from $H$ to
  $L^2(\Omega;E)$.
\end{thm}
\begin{proof}
  We are going to prove first that the Fr\'echet derivative of the
  solution map is G\^ateaux-differentiable, with G\^ateaux derivative
  equal to $Z^x:=(h,k) \mapsto Z^x_{hk}$, then we shall then show that
  $x \mapsto Z^x$ is continuous and bounded as a map from $H$ to
  $\cL_2(H;L^2(\Omega;E))$.

  \noindent%
  \textsl{Step 1.} Let $x \in H$ be arbitrary but fixed, and consider
  the family of maps $z^\varepsilon \in \cL_2(H;L^2(\Omega;E))$,
  indexed by $\varepsilon \in \erre$, defined as
  \[
  z^\varepsilon : (h,k) \longmapsto
  z^\varepsilon_{hk} := \frac1\eps\left(Y^{x+\eps k}_h- Y^x_h\right) - Z^x_{hk}.
  \]
  Elementary manipulations based on the equations satisfied by $Y^x$
  and $Z^x$ show that
  \[
    z_{hk}^\eps(t) + \int_0^t Az_{hk}^\eps(s)\,ds
    = \int_0^t \Bigl( \frac{g'(X^{\eps}) Y_h^{\eps} - g'(X)Y_h}{\eps}
      - g'(X)Z_{hk}-g''(X)Y_hY_k \Bigr)(s)\,ds,
  \]
  where the integrand on the right-hand side can be written as
  $R_\varepsilon + S_\varepsilon$, with
  \begin{align*}
    R_\eps &= \Bigl( \frac{g'(X^{\eps}) - g'(X)}{\eps}
    - g''(X)Y_k \Bigr) Y_h,\\
    S_\varepsilon &= \Bigl( g'(X^{\eps}) \frac{Y_h^{\eps}-Y_h}{\eps}
    -g'(X)Z_{hk} \Bigr).
  \end{align*}
  Further algebraic manipulations show that $R_\varepsilon =
  R'_\varepsilon + R''_\varepsilon$ and $S_\varepsilon =
  S'_\varepsilon + S''_\varepsilon$, where
  \begin{align*}
    R'_\varepsilon &:= \Bigl( \frac{g'(X + \eps Y_k)- g'(X)}{\eps} -
    g''(X)Y_k \Bigr) Y_h,\\
    R''_\varepsilon &:= \frac{g'(X^{\eps}) - g'(X+\eps Y_k)}{\eps} Y_h,\\
    S'_\varepsilon &:= g'(X)z^\varepsilon_{hk},\\
    S''_\varepsilon &:= \bigl( g'(X^{\eps}) - g'(X) \bigr)
    \frac{Y_h^{\eps} - Y_h}{\eps}.
  \end{align*}
  The integration-by-parts formula and obvious estimates yield
  \[
    \frac12 \norm{z_{hk}^\eps(t)}^2 + C\int_0^t
    \norm{z_{hk}^\eps(s)}_V^2\,ds \leq
    \norm{g}_{\lip} \int_0^t \norm{z^\varepsilon_{hk}(s)}_V^2\,ds +
    \int_0^t \ip[\big]{R_\eps + S''_\eps}{z_{hk}^\eps}(s)\,ds,
  \]
  Taking the supremum on both sides, one is left with, thanks to
  Young's inequality,
  \[
  \norm{z_{hk}^\eps}_{C([0,t];H)}^2 +
  \norm{z_{hk}^\eps}^2_{L^2(0,t;V)} \lesssim
  \delta \norm{z_{hk}^\eps}_{C([0,t];H)}^2
  + \int_0^t   \norm{z_{hk}^\eps}_{C([0,s];H)}^2\,ds
  + \frac{1}{\delta} \norm{R_\eps + S''_\varepsilon}_{L^1(0,T;H)}^2
  \]
  for all $\delta>0$, from which it follows, taking $\delta$
  sufficiently small and applying Gronwall's inequality,
  \[
  \E\norm{z_{hk}^\eps}_E^2 \lesssim
  \E\norm{R_\eps}_{L^1(0,T;H)}^2 + \E\norm{S''_\eps}_{L^1(0,T;H)}^2.
  \]
  We are going to show that the right-hand side tends to zero as
  $\varepsilon \to 0$. Since $g \in C^2_b$, it is evident that
  $R_\eps' \to 0$ almost everywhere as $\eps \to 0$ as well as that
  \[
  \abs{R'_\eps} \leq 2\norm{g''}_\infty \abs[\big]{Y_kY_h}.
  \]
  Since 
  \[
  \sup_{\norm{h}\leq 1} \norm{Y_hY_k}_{L^1(0,T;H)} \lesssim
  \sup_{\norm{h}\leq 1} \norm{Y_h}_{L^2(0,T;V)}
  \norm{Y_k}_{L^2(0,T;V)} \lesssim 
  \sup_{\norm{h} \leq 1} \norm{h} \norm{k} \leq \norm{k},
  \]
  the dominated convergence theorem yields
  \[
  \lim_{\varepsilon \to 0} \sup_{\norm{h}\leq 1}
  \E\norm[\big]{R_\eps'}_{L^1(0,T;H)}^2 = 0.
  \]
  Moreover, we have
  \[
  \abs{R''_\eps} \leq \norm{g''}_\infty 
  \abs[\bigg]{\frac{X^{x+\eps k} - X^x}{\eps} - Y_k^x}
  \abs{Y_h^x},
  \]
  so that
  \[
  \norm{R_\eps''}_{L^1(0,T; H)} \lesssim \norm[\bigg]{%
  \frac{X^{x+\eps k} - X^x}{\eps} - Y_k^x}_{L^2(0,T;V)}
  \norm{Y_h^x}_{L^2(0,T;V)},
  \]
  where $\norm{Y_h^x}_{L^2(0,T; V)} \lesssim \norm{h}$, hence, by
  Theorem~\ref{thm:d1},
  \[
  \sup_{\norm{h} \leq 1} \E\norm{R_\eps''}_{L^1(0,T;H)}^2
  \lesssim \E\norm[\bigg]{%
  \frac{X^{x+\eps k} - X^x}{\eps} - Y_k^x}^2_{L^2(0,T;V)} \to 0
  \]
  Finally, from
  \[
  \abs{S''_\eps} \leq \norm{g''}_\infty
  \abs[\bigg]{\frac{X^{x+\eps k}-X^x}{\eps}}
  \abs{Y_h^{x+\eps k}-Y_h^x}
  \]
  we deduce
  \[
  \norm{S_\eps'}_{L^1(0,T; H)} \lesssim \norm[\bigg]{%
  \frac{X^{x+\eps k} - X^x}{\varepsilon}}_{L^2(0,T;V)}
  \norm[\big]{Y_h^{x+\eps k} - Y_h^x}_{L^2(0,T;V)}.
  \]
  Since $(X^{x+\eps k} - X^x)/\varepsilon \to Y^x_k$ in $E$ as
  $\varepsilon \to 0$ and $x \mapsto Y_h^x$ is continuous from $H$ to
  $E$, we infer that $\norm{S_\eps'}_{L^1(0,T; H)} \to 0$. Moreover,
  it follows from
  \[
  \norm[\big]{Y_h^{x+\eps k} - Y_h^x}_{L^2(0,T;V)} \leq 2 \norm{h}
  \]
  that
  \[
  \sup_{\norm{h}\leq 1}\norm{S_\eps'}_{L^1(0,T;H)} \lesssim
  \norm[\Big]{\frac{X^{x+\eps k}-X^x}\eps}_{L^2(0,T;V)}.
  \]
  Recalling that, by Theorem~\ref{thm:d1}, $(X^{x+\eps k} -
  X^x)/\varepsilon \to Y^x_k$ in $L^2(\Omega;E)$ as $\varepsilon \to
  0$, this implies
  \[
  \lim_{\varepsilon \to 0} \sup_{\norm{h} \leq 1}
  \E\norm[\big]{S_\eps''}_{L^1(0,T;H)}^2 = 0.
  \]
  We thus conclude that
  \[
  \lim_{\varepsilon \to 0} \sup_{\norm h\leq 1}
  \norm[\big]{z_{hk}^\eps}_{L^2(\Omega;E)} = 0 
  \qquad\forall k \in H,
  \]
  i.e. the directional derivative of $x \mapsto Y^x:H \mapsto
  \cL(H,L^2(\Omega;E))$ exists for all directions and is given by the
  map $x \mapsto Z^x: H \to \cL_2(H;L^2(\Omega;E))$. Since we have
  already proved that $(h,k) \mapsto Z^x_{hk}$ is bilinear and
  continuous, we infer that $x\mapsto Y^x$ is G\^ateaux differentiable
  with derivative $Z^x$.

  \smallskip

  \noindent%
  \textsl{Step 2.} In order to conclude that $x \mapsto Y^x$ is
  Fr\'echet differentiable (with derivative necessarily equal to $Z$)
  it is enough to show, in view of a criterion already mentioned, that
  the map
  \begin{align*}
    x &\longmapsto Z^x\\
    H &\longrightarrow \cL_2(H;L^2(\Omega;E))
  \end{align*}
  is continuous. Let $(x_n)_n\subseteq H$ be a sequence converging to
  $x$ in $H$. We have, writing $Z^n$ in place of $Z^{x_n}$ for simplicity,
  \[
  (Z^n_{hk}-Z_{hk})' + A(Z^n_{hk}-Z_{hk}) = g'(X^n)Z^n_{hk} -
  g'(X)Z_{hk} + g''(X^n)Y^n_hY^n_k - g''(X)Y_hY_k,
  \]
  with initial condition $Z^n_{hk}(0)-Z_{hk}(0)=0$.  The right-hand
  side of the equation can be written as $R=\sum_{i\leq4} R_i$, with
  \begin{align*}
    R_1 &:= g'(X^n) ( Z_{hk}^n-Z_{hk} ), & 
    R_2 &:= ( g'(X^n)-g'(X) ) Z_{hk},\\
    R_3 &:= g''(X^n) ( Y_h^nY_k^n-Y_hY_k ), &
    R_4 &:= ( g''(X^n)-g''(X) ) Y_hY_k,
  \end{align*}
  so that, by the integration-by-parts formula,
  \[
    \frac12 \norm{Z^n_{hk}(t)-Z_{hk}(t)}^2 
    + C\int_0^t \norm{Z^n_{hk}(s)-Z_{hk}(s)}_V^2 \,ds
    \leq \int_0^t \ip[\big]{R}{Z^n_{hk}-Z_{hk}}(s)\,ds,
  \]
  where
  \[
    \int_0^t \ip{R_1}{Z^n_{hk}-Z_{hk}}(s)\,ds
    \leq \norm{g'}_\infty \int_0^t \norm{Z_{hk}^n(s)-Z_{hk}(s)}^2\,ds,
  \]
  and, for $i \neq 1$, by Young's inequality,
  \begin{align*}
    \int_0^t \ip{R_i}{Z^n_{hk}-Z_{hk}}(s)\,ds &\leq
    \norm[\big]{Z^n_{hk}-Z_{hk}}_{C([0,t];H)}
    \norm[\big]{R_i}_{L^1(0,t;H)}\\
    &\leq \delta \norm[\big]{Z^n_{hk}-Z_{hk}}^2_{C([0,t];H)} +
    \frac{1}{\delta} \norm[\big]{R_i}^2_{L^1(0,t;H)}.
  \end{align*}
  By an argument based on the Gronwall's inequality already used
  several times we obtain
  \[
    \norm[\big]{Z^n_{hk}-Z_{hk}}^2_E \lesssim
    \norm[\big]{R_2+R_3+R_4}^2_{L^1(0,T; H)},
  \]
  where $\norm{R_2} \leq \norm{g''}_\infty \norm{(X^n-X)Z_{hk}}$ and,
  by the bilinearity of $Z$,
  \[
  \norm[\big]{(X^n-X)Z_{hk}}_{L^1(0,T;H)} \lesssim
  \norm[\big]{X^n-X}_{L^2(0,T;V)} 
  \norm[\big]{Z_{hk}}_{L^2(0,T;V)} \lesssim \norm[\big]{X^n-X}_{L^2(0,T;V)} 
  \norm{h} \norm{k},
  \]
  from which it follows
  \[
  \sup_{\norm{h}, \norm{k}\leq
    1}\E\norm[\big]{(X^n-X)Z_{hk}}_{L^1(0,T;H)}^2
  \lesssim \norm[\big]{X^n-X}_{L^2(\Omega; L^2(0,T; V))} \to 0
  \]
  because $x\mapsto X^x$ is continuous from $H$ to
  $L^2(\Omega;E)$. Moreover, since $\norm{R_3} \leq \norm{g''}_\infty
  \norm{Y_h^nY_k^n-Y_hY_k}$, we have, recalling that $V \embed L^4$,
  \[
    \norm[\big]{R_3}_{L^1(0,T;H)} \leq
    \norm[\big]{Y_h^n - Y_h}_{L^2(0,T;V)}
    \norm[\big]{Y_k}_{L^2(0,T;V)}
    + \norm[\big]{Y_k^n - Y_k}_{L^2(0,T;V)}
    \norm[\big]{Y^n_h}_{L^2(0,T;V)},
  \]
  where both terms on the right-hand side tend to zero because $Y_h^n
  \to Y_h$ in $L^2(0,T;V)$ for all $h \in H$. The estimate
  \[
  \norm[\big]{Y_h^nY_k^n-Y_hY_k}_{L^1(0,T; H)}\lesssim \norm{h} \norm{k}
  \]
  then implies, by the dominated convergence theorem,
  \[
  \sup_{\norm{h},\norm{k} \leq 1} \E\norm[\big]{R_3}^2_{L^1(0,T;H)}
  \lesssim \sup_{\norm{h},\norm{k} \leq 1}
  \E\norm[\big]{Y_h^nY_k^n-Y_hY_k}_{L^1(0,T;H)}^2 \to 0.
  \]
  It remains to consider $R_4$: it is clear that
  $(g''(X^n)-g''(X))Y_hY_k\to 0$ almost everywhere by the continuity
  of $g''$, and, as before,
  \[
  \norm[\big]{(g''(X^n)-g''(X))Y_hY_k}_{L^1(0,T;H)} \lesssim \norm{g''}_\infty
  \norm{h} \norm{k},
  \]
  hence the dominated convergence theorem yields
  \[
  \sup_{\norm{h},\norm{k} \leq 1} \E\norm[\big]{R_4}^2_{L^1(0,T;H)}
  \to 0.
  \]
  We have thus proved that, as $n \to \infty$,
  \[
  \norm[\big]{Z^n-Z}_{\cL_2(H;L^2(\Omega;E))} 
  = \sup_{\norm{h},\norm{k} \leq 1}
  \norm[\big]{Z_{hk}^n-Z_{hk}}_{L^2(\Omega;E)} \to 0.
  \]
  Recalling that $x \mapsto Z^x$ is bounded on $H$, we conclude that
  $x \mapsto X^x$ is twice Fr\'echet-differentiable with continuous
  and bounded derivatives.
\end{proof}

%--------------------------------------------------------------------------

\ifbozza\newpage\else\fi
\section{Invariant measures}
\label{sec:inv}
Throughout this section, we consider equation \eqref{eq:0} with
$X_0 \in H$. Since all coefficients do not depend explicitly on
$\omega \in \Omega$, it follows by a standard argument that the
solution $X$ to \eqref{eq:0} is Markovian.  Let $P=(P_t)_{t \geq 0}$ be
the transition semigroup defined by
\[
  (P_t\varphi)(x) := \E \varphi(X^x(t)) \qquad \forall x\in H,
  \quad \varphi \in C_b(H).
\]
We shall assume from now on that the pair $(A,B)$ satisfies the
coercivity condition
\begin{equation}
  \label{eq:coerc}
  \ip{Ax}{x} \geq \frac12\norm{B(x)}^2_{\cL^2(U,H)} + C\norm{x}_V^2 - C_0
  \qquad \forall  x \in V,
\end{equation}
with $C_0>0$ a constant.

\begin{thm}
  \label{thm:exist}
  The set of invariant measures for the transition semigroup
  $(P_t)_{t \geq 0}$ is not empty.
\end{thm}
\begin{proof}
  Let $(X,\xi)$ be the unique strong solution to \eqref{eq:0}. For every
  $t \geq 0$ one has, by Proposition~\ref{prop:Ito},
  \begin{align*}
    &\frac12 \norm{X(t)}^2 + \int_0^t \ip{AX(s)}{X(s)}\,ds
      + \int_0^t\!\!\int_D \xi(s)X(s)\,ds\\
    &\hspace*{3em}= \frac12 \norm{x}^2
      + \frac12 \int_0^t \norm[\big]{B(X(s))}^2_{\cL^2(U,H)}\,ds
      + \int_0^t X(s)B(X(s))\,dW(s).
  \end{align*}
  Let us show that the stochastic integral $M:= XB(X) \cdot W$ on the
  right-hand side of this identity is a martingale. For this it
  suffices to show that $\E[M,M]_T^{1/2}$ is finite: one has, by the
  ideal property of Hilbert-Schmidt operators and the Cauchy-Schwarz
  inequality,
  \begin{align*}
    \E[M,M]_T^{1/2}
    &= \E\biggl(
    \int_0^T \norm[\big]{XB(X)}^2_{\cL^2(U,\erre)}\,ds \biggr)^{1/2}\\
    &\leq \E\norm[\big]{X}_{L^\infty(0,T;H)} \biggl(
      \int_0^T \norm[\big]{B(X)}^2_{\cL^2(U,H)}\,ds \biggr)^{1/2}\\
    &\leq \Bigl( \E\norm[\big]{X}^2_{L^\infty(0,T;H)} \Bigr)^{1/2}
      \biggl(\E\int_0^T \norm[\big]{B(X)}^2_{\cL^2(U,H)}\,ds \biggr)^{1/2},
  \end{align*}
  where the last term is finite thanks to Theorem~\ref{thm:WP} and the
  assumption of linear growth on $B$. Therefore, recalling that, for
  any $r,s \in \erre$, $j(r) + j^*(s)=rs$ if and only if
  $s \in \beta(r)$, one has, taking the coercivity condition
  \eqref{eq:coerc} into account,
  \begin{equation}
    \label{eq:est1}
    C \E\int_0^t\norm{X(s)}^2_V\,ds +
    \E\int_0^t\!\!\int_D j(X(s)) \,ds + \E\int_0^t\!\!\int_D
    j^*(\xi(s)) \,ds \leq \frac12\norm{x}^2 + C_0t
  \end{equation}
  for all $t \geq 0$.
  Let $x=0$. For any $t \geq 0$ the law of the random variable $X(t)$
  is a probability measure on $H$, which we shall denote by
  $\pi_t$. We are now going to show that the family of measures
  $(\mu_t)_{t>0}$ on $H$ defined by
  \[
  \mu_t: E \longmapsto \frac1t \int_0^t \pi_s(E)\,ds
  \]
  is tight. The ball $B_n$ in $V$ of radius $n\in\mathbb{N}$ is a
  compact subset of $H$, because the embedding $V \embed H$ is
  compact. Moreover, Markov's inequality and \eqref{eq:est1} yield
  \begin{align*}
    \mu_t(B_n^c)
    &= \frac1t\int_0^t\pi_s(B_n^c)\,ds
      = \frac1t\int_0^t\P\bigl( \norm{X(s)}_V^2 > n^2 \bigr) \,ds\\
    &\leq \frac1{tn^2}\int_0^t\E\norm{X(s)}_V^2\,ds \leq
      \frac1{Ctn^2} \, C_0 t = \frac{C_0}{Cn^2},
  \end{align*}
  hence also
  \[
  \sup_{t>0}\mu_t(B_n^c) \leq \frac{C_0}{Cn^2} \to 0 \quad\text{as } n\to\infty.
  \]
  It follows by Prokhorov's theorem that there exists a probability
  measure $\mu$ on $H$ and a sequence $(t_k)_{k\in\enne}$ increasing
  to infinity such that $\mu_{t_k}$ converges to $\mu$ in the topology
  $\sigma(\mathscr{M}_1(H),C_b(H))$ as $k \to \infty$. Furthermore,
  $\mu$ is an invariant measure for the transition semigroup $P$,
  thanks to the Krylov-Bogoliubov theorem.
\end{proof}

We are now going to prove integrability properties of \emph{all}
invariant measures, which in turn provide information on their
support. We start with a (relatively) simple yet crucial estimate.
\begin{prop}
  Let $\mu$ be an invariant measure for the transition semigroup
  $(P_t)$. Then one has
  \[
  \int_H \norm{x}^2\,\mu(dx) \leq \frac{K^2C_0}{C},
  \]
  where $K$ is the norm of the embedding $V \embed H$.
\end{prop}
\begin{proof}
  We are going to apply the It\^o formula of
  Proposition~\ref{prop:Ito_pard} to the process $X$ and the function
  $G_\delta: x \mapsto g_\delta\bigl(\norm{x}^2\bigr)$, where
  $g_\delta \in C^2_b(\erre_+)$ is defined as
  \[
    g_\delta(r)=\frac{r}{1+\delta r}, \qquad \delta>0,
  \]
  so that
  \[
    g'_\delta(r)=\frac{1}{(1+\delta r)^2}, \quad
    g''_\delta(r)=-\frac{2\delta}{(1+\delta r)^3},
  \]
  whence
  \begin{align*}
    &g_\delta\bigl(\norm{X(t)}^2\bigr)
    + 2 \int_0^t g'_\delta\bigl(\norm{X(s)}^2\bigr)%
      \bigl( \ip{AX(s)}{X(s)} + \ip{\xi(s)}{X(s)} \bigr)\,ds\\
    &\hspace{4em} - 2 \int_0^t g''_\delta\bigl(\norm{X(s)}^2\bigr)%
      \norm[\big]{X(s)B(X(s))}^2_{\cL^2(U,\erre)}\,ds\\
    &\qquad = g_\delta\bigl(\norm{x}^2\bigr)
      + 2\int_0^t g'_\delta\bigl(\norm{X(s)}^2\bigr)%
      X(s)B(X(s))\,dW(s)\\
    &\hspace{4em} + \int_0^t g'_\delta\bigl(\norm{X(s)}^2\bigr)%
      \norm[\big]{B(X(s))}^2_{\cL^2(U,\erre)}\,ds.
  \end{align*}
  Since $g'_\delta>0$ and $g''_\delta<0$, the coercivity condition
  \eqref{eq:coerc} and the monotonicity of $\beta$ imply
  \begin{align*}
    &g_\delta\bigl(\norm{X(t)}^2\bigr)
      + 2 \int_0^t g'_\delta\bigl(\norm{X(s)}^2\bigr)%
      \bigl( C\norm{X(s)}^2_V - C_0 \bigr)\,ds\\
    &\hspace{4em} \leq g_\delta\bigl(\norm{x}^2\bigr)
      + 2\int_0^t g'_\delta\bigl(\norm{X(s)}^2\bigr)%
      X(s)B(X(s))\,dW(s).
  \end{align*}
  Taking into account that $\abs{g'_\delta} \leq 1$, the
  stochastic integral is a martingale, exactly as in the proof of
  Theorem~\ref{thm:exist}, hence has expectation zero, so that
  \[
    \E G_\delta(X(t))
    + 2C \E \int_0^t g'_\delta\bigl(\norm{X(s)}^2\bigr) \norm{X(s)}^2_V\,ds
    \leq G_\delta(x) + 2Ct.
  \]
  By definition of $(P_t)$ we have
  $P_tG_\delta(x) = \E G_\delta(X(t))$, from which it follows, by
  the boundedness of $G_\delta$ and by definition of invariant measure,
  \[
    C\int_H \E\int_0^t g'_\delta\bigl(\norm{X(s)}^2\bigr) \norm{X(s)}^2_V
    \,ds\,\mu(dx) \leq C_0t.
  \]
  Denoting the norm of the embedding $V \embed H$ by $K$, we get
  \[
    \int_H\!\int_0^t\E\frac{\norm{X(s)}^2}%
    {\bigl( 1+\delta \norm{X(s)}^2\bigr)^2}\,ds\,d\mu
    \leq \frac{K^2C_0}{C}t.
  \]
  Let $f_\delta:r \mapsto r/(1+\delta r)^2$, $\delta>0$, and
  $F_\delta:= f_\delta \circ \norm{\cdot}^2$. Then
  \[
    \E\frac{\norm{X(s)}^2}{\bigl( 1+\delta \norm{X(s)}^2\bigr)^2}
    = P_sF_\delta,
  \]
  hence, by Tonelli's theorem and invariance of $\mu$,
  \[
  \int_H\!\int_0^t\E\frac{\norm{X(s)}^2}%
  {\bigl( 1+\delta \norm{X(s)}^2\bigr)^2}\,ds\,d\mu
  = \int_0^t \int_H P_sF_\delta\,d\mu\,ds
  = t \int_H F_\delta\,d\mu
  \leq \frac{K^2C_0}{C}t.
  \]
  Taking the limit as $\delta \to 0$, the monotone convergence theorem
  yields
  \[
    \int_H \norm{x}^2\,\mu(dx) \leq \frac{K^2C_0}{C}.
    \qedhere
  \]
\end{proof}

In order to state the next integrability results for invariant
measures, we need to define the following subsets of $H$:
\begin{align*}
  J &:= \bigl\{ u\in H: j(u)\in L^1(D) \bigr\},\\
  J^* &:= \bigl\{ u\in H: \exists\, v\in L^1(D): v\in\beta(u)
        \text{ a.e.~in } D \text{ and } j^*(v)\in L^1(D) \bigr\},
\end{align*}
whose Borel measurability will be proved in Lemma~\ref{lm:meas} below.
\begin{thm}  \label{thm:supp}
  Let $\mu$ be an invariant measure for the transition semigroup
  $P$. Then one has
  \[
    \int_H\norm{u}_V^2\,\mu(du) + \int_H\int_Dj(u)\,\mu(du) +
    \int_H\int_Dj^*(\beta^0(u))\,\mu(du)\leq \frac{K^2C_0}{2C}+C_0,
  \]
  where $K$ is the norm of the embedding $V \embed H$. In particular,
  $\mu$ is concentrated on $V \cap J \cap J^*$.
\end{thm}
\begin{proof}
  Let us introduce the functions $\Phi$, $\Psi$,
  $\Psi_*:H \to \erre_+ \cup \{+\infty\}$ defined as
  \begin{align*}
    \Phi: u &\longmapsto \norm{u}_V^2 1_V(u)
              + \infty \cdot 1_{H \setminus V}(u),\\
    \Psi: u &\longmapsto \Bigl( \int_D j(u) \Bigr) 1_J(u)
              + \infty \cdot 1_{H \setminus J}(u),\\
    \Psi_*: u &\longmapsto \Bigl( \int_D j^*(\beta^0(u)) \Bigr) 1_{J^*}(u)
              + \infty \cdot 1_{H \setminus J^*}(u),
  \end{align*}
  as well as their approximations $\Phi_n$, $\Psi_n$,
  $\Psi_{*n}:H \to \erre_+ \cup \{+\infty\}$, $n \in \enne$, defined
  as (here $B_n(V)$ denotes the ball of radius $n$ in $V$)
  \[
  \Phi_n: u \longmapsto \begin{cases}
  \norm{u}_V^2 \quad&\text{if } u \in B_n(V),\\[4pt]
  n^2 &\text{if } u\in H\setminus B_n(V),
  \end{cases}
  \]
  \[
  \Psi_n: u \longmapsto \begin{cases}
  \int_Dj(u) \quad&\text{if } \int_Dj(u)\leq n,\\[4pt]
  n &\text{otherwise},
  \end{cases}
  \]
  and
  \[
  \Psi_{*n}: u \longmapsto \begin{cases}
    \int_Dj^*(\beta_{1/n}(u)) \quad
    &\text{if } \int_Dj^*(\beta_{1/n}(u))\leq n,\\[4pt]
  n &\text{otherwise}.
  \end{cases}
  \]
  One obviously has
  \[
  \int_H \Phi_n \,d\mu = \int_0^1\!\!\int_H\Phi_n\,d\mu\,ds,
  \]
  as well as, by invariance of $\mu$ and boundedness of $\Phi_n$,
  \[
  \int_H \Phi_n\,d\mu = \int_H P_s\Phi_n\,d\mu,
  \]
  thus also, by Tonelli's theorem ($\Phi_n \geq 0$ and $P$ is
  positivity preserving, being Markovian)
  \[
    \int_H \Phi_n\,d\mu
    = \int_0^1\!\!\int_H P_s\Phi_n\,d\mu\,ds
    = \int_H\!\int_0^1 \E\Phi_n(X(s))\,ds\,d\mu.
  \]
  The same reasoning also yields
  \[
    \int_H \Psi_n\,d\mu
    = \int_H\!\int_0^1 \E\Psi_n(X(s))\,ds\,d\mu,
    \qquad
    \int_H \Psi_{*n}\,d\mu
    = \int_H\!\int_0^1 \E\Psi_{*n}(X(s))\,ds\,d\mu,
  \]
  with
  \begin{align*}
    \E\Phi_n(X(s))
    &= \E\bigl( \norm{X(s)}_V^2 \wedge n^2 \bigr) \leq \E \norm{X(s)}_V^2,\\
    \E\Psi_n(X(s))
    &= \E \Bigl( n \wedge \int_D j(X(s)) \Bigr) \leq \E\int_D j(X(s)),\\
    \E\Psi_{*n}(X(s))
    &= \E\Bigl( n \wedge \int_D j^*(\beta_{1/n}(X(s))) \Bigr)
      \leq \E\int_D j^*(\xi(s)),
  \end{align*}
  where, in the last inequality, we have used the fact that for every
  $r\in\dom(\beta)=\erre$ the sequence $\{\beta_\lambda(r)\}_\lambda$
  converges from below to $\beta^0(r)$, where $\beta^0(r)$ is the
  unique element in $\beta(r)$ such that $|\beta^0(r)|\leq|y|$ for
  every $y\in\beta(r)$ (note that the uniqueness of $\beta^0(r)$
  follows from the maximal monotonicity of $\beta$).  Thanks to
  estimate \eqref{eq:est1} we have, by Tonelli's theorem,
  \begin{align*}
    & C \int_0^1 \bigl( \E\Phi_n(X(s)) + \E\Psi_n(X(s)) + \E\Psi_{*n}(X(s))
      \bigr)\,ds\\
    &\hspace{3em} \leq C \E\int_0^1 \norm{X(s)}_V^2\,ds
      + \E\int_0^1\!\!\int_D j(X(s))\,ds
      + \E\int_0^1\!\!\int_D j^*(\xi(s))\,ds\\
    &\hspace{3em} \leq \frac12 \norm{x}^2 + C_0,
  \end{align*}
  therefore, integrating with respect to $\mu$ and taking the previous
  proposition into account,
  \[
    \int_H \bigl(C\Phi_n + \Psi_n + \Psi_{*n}\bigr)\,d\mu
    \leq \frac{1}{2}\int_H\norm{x}^2\,\mu(dx)
    + C_0 \leq \frac{K^2C_0}{2C}+C_0
  \]
  uniformly with respect to $n$. Since $\Phi_n$ and $\Psi_n$ converge
  pointwise and monotonically from below to $\Phi$ and $\Psi$,
  respectively, the monotone convergence theorem yields
  \[
    \int_H \Phi\,d\mu \leq \frac{C_0(K^2 + 2C)}{2C^2},
    \qquad
    \int_H \Psi\,d\mu \leq \frac{C_0(K^2 + 2C)}{2C},
  \]
  hence, in particular, $\mu(V)=\mu(J)=1$. Similarly, note that
  $\beta_{1/n} \in \beta((I+(1/n)\beta)^{-1})$ and $0 \in \beta(0)$
  imply that $\abs{\beta_{1/n}}$ converges pointwise to
  $\abs{\beta^0}$ monotonically from below as $n \to \infty$, hence
  the same holds for the convergence of $j^*(\beta_{1/n})$ to
  $j^*(\beta^0)$ because $j^*$ is convex and continuous with
  $j^*(0)=0$. Therefore $\Psi_{*n}$ converges to $\Psi$ pointwise
  monotonically from below as $n \to \infty$. We conclude, again by
  the monotone convergence theorem, that $\Psi_* \in L^1(H,\mu)$, thus
  also that $\mu(J^*)=1$.
\end{proof}
As mentioned above, the sets $J$ and $J^*$ are Borel measurable.
\begin{lemma}   \label{lm:meas}
  The sets
  \begin{align*}
  J &:= \bigl\{ u\in H: j(u) \in L^1(D) \bigr\},\\
  J^* &:= \bigl\{ u\in H: \exists\, v\in L^1(D): v\in\beta(u)
        \text{ a.e.~in } D \text{ and } j^*(v)\in L^1(D) \bigr\},
  \end{align*}
  belong to the Borel $\sigma$-algebra of $H$.
\end{lemma}
\begin{proof}
  Setting, for every $n \in \enne$,
  \begin{align*}
  J_n &:= \bigl\{ u\in H: \int_Dj(u)\leq n \bigr\},\\
    J^*_n &:= \bigl\{ u\in H: \exists\, v\in L^1(D): v\in\beta(u)
            \text{ a.e.~in } D \text{ and } \int_Dj^*(v)\leq n \bigr\},
  \end{align*}
  it is immediately seen that
  \[
  J=\bigcup_{n=1}^\infty J_n \quad\text{and}\quad
  J^*=\bigcup_{n=1}^\infty J^*_n.
  \]
  Moreover, the lower semicontinuity of convex integrals implies that
  $J_n$ is closed in $H$ for every $n$, hence Borel-measurable, so
  that $J \in \cB(H)$. Let us show that, similarly, $J_n^*$ is also
  closed in $H$ for every $n\in\enne$: if $(u_k)_k\subset J_n^*$ and
  $u_k\to u$ in $H$, then for every $k$ there exists $v_k \in L^1(D)$
  with $v_k\in\beta(u_k)$ and
  \[
    \int_Dj^*(v_k) \leq n \qquad \forall k \in \enne.
  \]
  Since $j^*$ is superlinear at infinity, this implies that the family
  $(v_k)_k$ is uniformly integrable in $D$, hence by the
  Dunford-Pettis theorem also weakly relatively compact in
  $L^1(D)$. Consequently, there is a subsequence $(v_{k_i})_i$ and
  $v\in L^1(D)$ such that $v_{k_i} \to v$ weakly in $L^1(D)$.  The
  weak lower semicontinuity of convex integrals easily implies
  that
  \[
    \int_Dj^*(v)\leq\liminf_{i\to\infty}\int_Dj^*(v_{k_i})\leq n\,.
  \]
  Let us show that $v\in\beta(u)$ almost everywhere in $D$: by
  definition of subdifferential, for every $k \in \enne$ and for every
  measurable set $E \subseteq D$ we have
  \[
    \int_Ej(u_k) + \int_E v_k(z-u_k) \leq \int_Ej(z)
    \qquad\forall z\in L^\infty(D).
  \]
  By Egorov's theorem, for any $\varepsilon>0$ there exists a
  measurable set $E_\varepsilon \subseteq D$ with
  $|E_\varepsilon^c| \leq \varepsilon$ and $u_k \to u$ uniformly in
  $E_\varepsilon$. Taking $E=E_\varepsilon$ in the last inequality,
  letting $k \to \infty$ we get
  \[
    \int_{E_\varepsilon}j(u) + \int_{E_\varepsilon} v(z-u) \leq
    \int_{E_\varepsilon}j(z) \quad\forall z\in L^\infty(D),
  \]
  which in turn implies by a classical localization argument that
  \[
    j(u)+v(z-u)\leq j(z) \quad\text{a.e.~in } E_\varepsilon,
    \qquad \forall z\in\erre.
  \]
  Hence, by the arbitrariness of $\varepsilon$, $v \in \beta(u)$
  almost everywhere in $D$, thus also $u\in J^*_n$.  This implies that
  $J^*_n$ is closed in $H$ for every $n$, therefore also that
  $J^* \in \cB(H)$.
\end{proof}

The estimates proved above implies that the set of ergodic invariant
measures is not empty.
\begin{thm}
  \label{th:3}
  There exists an ergodic invariant measure for the transition
  semigroup $(P_t)$.
\end{thm}
\begin{proof}
  Recall that, as it follows by the Krein-Milman theorem, for a
  Markovian transition semigroup the set of ergodic invariant measures
  coincides with the extreme points of the set of all invariant
  measures (see, e.g., \cite[Thm.~19.25]{AliBor}). Let $\mathscr{I}$
  be the set of all invariant measures for $P$: by Theorem
  \ref{thm:exist}, we know that $\mathscr{I}$ is not empty and we need
  to show that $\mathscr{I}$ admits at least an extreme point. Let us
  prove that $\mathscr{I}$ is tight.  By Theorem \ref{thm:supp}, we
  know that there exists a constant $N$ such that
  \[
    \int_H\norm{x}^2_V \,\mu(dx) \leq N
    \qquad\forall \mu\in\mathscr{I}.
  \]
  Therefore, using the notation of the proof of Theorem
  \ref{thm:exist}, by Markov inequality
  \[
    \sup_{\mu\in\mathscr{I}}\mu(B_n^c) =
    \sup_{\mu\in\mathscr{I}} \mu\bigl(\{x\in H: \norm{x}_V >n\}\bigr)
    \leq \frac{1}{n^2} \sup_{\mu \in \mathscr{I}} \int_H \norm{x}_V^2 \mu(dx)
    \leq \frac{N}{n^2} \to 0
  \]
  as $n\to\infty$. Hence $\mathscr{I}$ is tight, and thus admits
  extreme points.
\end{proof}

Under a very mild growth condition on the drift one can also obtain
uniqueness.
\begin{thm}
  \label{th:4}
  If $\beta$ is superlinear, i.e. if there exists $c>0$ and $\delta>0$
  such that
  \[
    (y_1-y_2)(x_1-x_2) \geq c|x_1-x_2|^{2+\delta}
    \qquad\forall (x_i,y_i) \in \beta, \quad i=1,2,
  \]
  then there exists a unique invariant measure $\mu$ for the
  transition semigroup $P$.  Moreover, $\mu$ is strongly mixing.
\end{thm}
\begin{proof}
  For any $x,y\in H$, by It\^o's formula, the monotonicity of $A$, the
  superlinearity of $\beta$, and Jensen inequality we have
  \[
    \E\norm{X(t; 0,x)-X(t;0,y)}^2 + \tilde
    c\int_0^t\left(\E\norm{X(t;
        0,x)-X(t;0,y)}^2\right)^{1+\frac\delta2} \leq \norm{x-y}^2
  \]
  for a positive constant $\tilde c$. Denoting by $y(\cdot;y_0)$ the
  solution to the Cauchy problem
  \[
  y'+y^{1+\frac\delta2}=0, \qquad y(0)=y_0\geq0,
  \]
  one can easily check that 
  \[
  c(t):=\sup_{y_0\geq 0}y(t;y_0) \to 0 \quad\text{as } t\to\infty
  \]
  and that $c(t)\geq0$ for every $t\geq0$. We deduce that 
  \[
    \E\norm[\big]{X(t; 0,x)-X(t;0,y)}^2 \leq c(t)
    \qquad\forall t\geq0.
  \]
  Let $\mu$ be an invariant measure for $P$. For any
  $\varphi\in C^1_b(H)$ we have
  \[
  \left|P_t\varphi(x) - \int_H\varphi(y)\,\mu(dy)\right|^2 \leq 
  \norm{D\varphi}_\infty^2\int_H \E\norm{X(t; 0,x)-X(t;0,y)}^2\,\mu(dy)\leq
  \norm{D\varphi}_\infty^2 c(t)
  \]
  uniformly in $x$, and since $C^1_b(H)$ is dense in $L^2(H,\mu)$, we
  deduce that for any $x\in H$
  \[
    \left|P_t\varphi(x) - \int_H\varphi(y)\,\mu(dy)\right|^2 \to 0
  \]
  as $t \to \infty$ for every $\varphi \in L^2(H,\mu)$. We have thus
  shown that $P$ admits a unique invariant measure, which is strongly
  mixing as well.
\end{proof}

%--------------------------------

\ifbozza\newpage\else\fi
\section{The Kolmogorov equation}
\label{sec:kolm}
Throughout this section we shall assume that $\beta$ is a function,
rather than just a graph.

Let $P=(P_t)_{t \geq 0}$ be the Markovian semigroup on $B_b(H)$
generated by the unique solution to \eqref{eq:0}, as in the previous
section, and $\mu$ be an invariant measure for $P$. Then $P$ extends to a
strongly continuous linear semigroup of contractions on $L^p(H,\mu)$
for every $p \geq 1$. These extensions will all be denoted by the same
symbol.  Let $-L$ be the infinitesimal generator in $L^1(H,\mu)$ of
$P$, and $-L_0$ be Kolmogorov operator formally associated to
\eqref{eq:0}, i.e.
\[
  [L_0f](x) =
  - \frac12 \operatorname{Tr}\bigl( D^2f(x) B(x)B^*(x) \bigr)
  + \ip{Ax}{Df(x)} + \ip{\beta(x)}{Df(x)}, \qquad x \in V\cap J^*,
\]
where $f$ belongs to a class of sufficiently regular functions
introduced below. Our aim is to characterize the ``abstract''
operator $L$ as the closure of the ``concrete'' operator $L_0$. Even
though this will be achieved only in the case of additive noise, some
intermediate results will be proved in the more general case of
multiplicative noise.

Let us first show that $L_0$ is a proper linear (unbounded) operator
on $L^1(H,\mu)$ with domain
\[
  \dom(L_0) := \bigl\{ f:V \to \erre:\, f \in C^1_b(V') \cap C^2_b(H)
  \cap C^1_b(L^1(D)) \bigr\}.
\]
Here $f \in C^1_b(V')$ means that, for any $x \in V$ and $v' \in V'$,
$\abs{Df(x)v'} \leq N\norm{v'}_{V'}$, with the constant $N$
independent of $x$ and $v'$, and that $x \mapsto Df(x) \in
C(V',V)$. Analogously, $f \in C^1_b(L^1(D))$ means that, for any
$x \in V$ and $k \in L^1(D)$, there is a constant $N$ independent of
$x$ and $k$ such that $\abs{Df(x)k} \leq N\norm{k}_{L^1(D)}$ and
$x \mapsto Df(x) \in C(L^1(D),L^\infty(D))$.
For any $f \in C^2_b(H)$, one has, recalling the linear
growth condition on $B$,
\[
  \operatorname{Tr}\bigl( D^2f(x)B(x)B^*(x) \bigr)
  \lesssim \norm{B(x)}^2_{\cL^2(U,H)} \lesssim 1 + \norm{x}^2,
\]
and $\norm{\cdot}^2 \in L^1(H,\mu)$. Moreover, since
$A \in \cL(V,V')$, one has $\norm{Ax}_{V'} \lesssim \norm{x}_V$, so
that, for any $f \in C^1_b(V')$,
\[
  \abs[\big]{\ip{Ax}{Df(x)}} \leq \norm{Ax}_{V'}
  \, \sup_{x\in V} \norm{Df(x)}_V
  \lesssim \norm{x}_{V},
\]
hence $x \mapsto \ip{Ax}{Df(x)} \in L^1(H,\mu)$ because
$\norm{\cdot}_V^2 \in L^1(H,\mu)$. Similarly, writing
\[
  \abs[\big]{\ip{\beta(x)}{Df(x)}} \leq
  \norm[\big]{j^*(\beta(x))}_{L^1(D)}
  + \norm[\big]{j(Df(x))}_{L^1(D)}
\]
and recalling that
$x \mapsto \norm{j^*(\beta(x))}_{L^1(D)} \in L^1(H,\mu)$ by
Theorem~\ref{thm:supp}, it is enough to consider the second term on
the right-hand side: for any $f \in C^1_b(L^1(D))$,
$\sup_{x \in V} \norm{Df(x)}_{L^\infty(D)}$ is finite, hence,
recalling that $j \in C(\erre)$, $j(Df(x))$ is bounded pointwise in
$D$, thus also in $L^1(D)$, uniformly over $x \in V$. In particular,
$x \mapsto \norm{j(Df(x))}_{L^1(D)} \in L^1(H,\mu)$.

Let us now show that the infinitesimal generator $-L$ restricted to
$\dom(L_0)$ coincides with the operator $-L_0$ defined above.  Indeed,
by Proposition~\ref{prop:Ito_pard2}, for every $g\in\dom(L_0)$ we have
\begin{align*}
  &g(X^x(t)) +
  \int_0^t\ip{AX^x(s)}{Dg(X^x(s))}\,ds
  +\int_0^t\ip{\beta(X^x(s))}{Dg(X^x(s))}\,ds\\
  &\qquad =g(x)%
    + \frac12\int_0^t\operatorname{Tr}[B^*(X^x(s))D^2g(X^x(s))B(X^x(s))]\,ds\\
  &\qquad\quad +\int_0^tDg(X^x(s))B(X^x(s))\,dW(s),
\end{align*}
from which we infer, taking expectations and applying Fubini's theorem,
\[
  \frac{P_tg(x)-g(x)}t=
  -\frac1t\int_0^tP_sL_0g(x)\,ds \qquad\forall\,x\in V\cap J^*.
\]
Since $g\in\dom(L_0)$, we have that $L_0g\in L^1(H,\mu)$, as proved
above. Therefore, recalling that $P$ is strongly continuous on
$L^1(H,\mu)$, we have that $t\mapsto P_tL_0g$ is continuous from
$[0,T]$ to $L^1(H,\mu)$. Hence, letting $t \to 0$, we have
\[
  \frac{P_tg-g}t \to -L_0g \qquad\text{in } L^1(H,\mu),
\]
which implies that $L=L_0$ on $\dom(L_0)$.

We are now going to construct a regularization of the operator $L_0$.
For any $\lambda \in (0,1)$, let
\[
\beta_\lambda:\erre \to \erre, \qquad 
\beta_\lambda := \frac1\lambda \bigl(I - (I+\lambda\beta)^{-1} \bigr),
\]
be the Yosida approximation of $\beta$.  Denoting a sequence of
mollifiers on $\erre$ by $(\rho_n)$, the function $\beta_{\lambda
  n}:=\beta_\lambda \ast \rho_n$ is monotone and infinitely
differentiable with all derivatives bounded.
Let us consider the regularized equation
\begin{equation}
\label{eq:reg}
dX_{\lambda n} + AX_{\lambda n}\,dt + \beta_{\lambda n}(X_{\lambda n})\,dt 
= B(X_{\lambda n})\,dW(t),
\qquad X_{\lambda n}(0)=x.
\end{equation}
Since $\beta_{\lambda n}$ is Lipschitz-continuous, equation \eqref{eq:reg} admits a unique
strong (variational) solution $X_{\lambda n}^x \in L^2(\Omega;E)$,
where, as before, $E:=C([0,T];H) \cap L^2(0,T;V)$.
Furthermore, the generator of the Markovian transition semigroup
$P^{\lambda n}=(P^{\lambda n}_t)_{t \geq 0}$ on $B_b(H)$ defined by
$P^{\lambda n}_t f(x) := \E f(X^x_{\lambda n}(t))$, restricted to
$C^1_b(V') \cap C^2_b(H)$, is given by
$-L_0^{\lambda n}$, where
\[
  [L_0^{\lambda n}f](x) =
  - \frac12 \operatorname{Tr}\bigl( D^2f(x) B(x)B^*(x) \bigr)
  + \ip{Ax}{Df(x)} + \ip{\beta_{\lambda n}(x)}{Df(x)},
\qquad x \in V.
\]
This follows arguing as in the case of $L_0$ (even using the simpler
It\^o formula of Proposition~\ref{prop:Ito_pard}, rather than the one
of Proposition~\ref{prop:Ito_pard2}).

Let us now consider the stationary Kolmogorov equation
\begin{equation}
\label{eq:Kreg}
\alpha v + L_0^{\lambda n} v = g, \qquad g \in \dom(L_0),
\quad \alpha>0.
\end{equation}
In view of the well-known relation between (Markovian) resolvents and
transition semigroups, one is led to considering the function
\[
v_{\lambda n}(x) := \E\int_0^\infty e^{-\alpha t} g(X^x_{\lambda n}(t))\,dt,
\]
which is the natural candidate to solve \eqref{eq:Kreg}. If we show
that $v_{\lambda n} \in C^1_b(V') \cap C^2_b(H)$, then an application
of It\^o's formula (in the version of Proposition~\ref{prop:Ito_pard})
shows that indeed $v_{\lambda n}$ solves \eqref{eq:Kreg}.
We are going to obtain regularity properties of $v_{\lambda n}$ via
pathwise differentiability of the solution map $x \mapsto X_{\lambda
  n}$ of the regularized stochastic equation \eqref{eq:reg}. From now
on we shall restrict our considerations to the case of additive noise,
i.e. we assume that $B \in \cL^2(U,H)$ is non-random. Moreover, we
shall assume that $V$ is continuously embedded in $L^4(D)$. The latter
assumption is needed to apply the second-order differentiability
results of \S\ref{ssec:diff}.
We recall that, thanks to Theorems~\ref{thm:d1} and \ref{thm:d2}, the
solution map $x \mapsto X_{\lambda n}: H \to L^2(\Omega;E)$ is
Lipschitz continuous and twice Fr\'echet differentiable. Moreover,
denoting its first order Fr\'echet differential by
\[
DX_{\lambda n} : H \to \cL(H,L^2(\Omega;E)),
\]
for any $h \in H$ the process $Y_h:=(DX_{\lambda n})h \in
L^2(\Omega;E)$ satisfies the linear deterministic equation with random
coefficients
\begin{equation}
\label{eq:rde}
Y'_h(t) + A Y_h(t) 
+ \beta'_{\lambda n}(X_{\lambda n}(t)) Y_h(t) = 0, 
\qquad Y_h(0) = h.
\end{equation}
Similarly, denoting the second order Fr\'echet differential of $x
\mapsto X_{\lambda n}$ by
\[
  D^2X_{\lambda n} :H \to \cL_2(H;L^2(\Omega;E)),
\]
for any $h,k \in H$ the process $Z_{hk}:=D^2X_{\lambda n}(h,k) \in
L^2(\Omega;E))$ satisfies the linear deterministic equation with
random coefficients
\begin{equation}
  \label{eq:rde'}
  Z'_{hk}(t) + AZ_{hk}(t) 
  + \beta'_{\lambda n}(X_{\lambda n}(t)) Z_{hk}(t)
  + \beta_{\lambda n}''(X_{\lambda n}(t))Y_h(t) Y_k(t) = 0, 
  \qquad Z_{hk}(0) = 0.
\end{equation}

We shall need the following result on the connection between
variational and mild solutions in the deterministic setting. We recall
that $A_2$ denotes the part of $A$ on $H$.
\begin{lemma}   \label{lm:var-mild}
  Let $F:[0,T]\times H\to H$ be Lipschitz continuous in the second
  variable, uniformly with respect to the first, with $F(\cdot, 0)=0$,
  and $u_0 \in H$.  If $u\in C([0,T];H) \cap L^2(0,T;V)$ and $v\in
  C([0,T];H)$ are the (unique) variational and mild solution to the
  problems
  \[
  u'+Au=F(\cdot,u), \quad u(0)=u_0, \qquad \text{ and } \qquad
  v'+A_2v=F(\cdot,v), \quad v(0)=u_0,
  \]
  respectively, then $u=v$.
\end{lemma}
\begin{proof}
  Let us first assume that 
  $u'+Au=f$ and $v'+A_2v=f$, where
  $f\in L^2(0,T; H)$. Then we have
  \begin{gather*}
    u(t) + \int_0^tAu(s)\,ds = u_0 + \int_0^tf(s)\,ds,\\
    v(t) = S(t)u_0 + \int_0^tS(t-s)f(s)\,ds
  \end{gather*}
  for all $t \in [0,T]$, where $S$ is the the semigroup generated on
  $H$ by $-A_2$. Let us show that $u=v$.  For $m\in\enne$, applying
  $(I+\eps A_2)^{-m}$ to the second equation we have (with obvious
  meaning of notation)
  \[
  v_\eps' + A_2 v_\eps = f_\eps, \qquad v_\eps(0)=u_0^\eps
  \]
  in the strong sense, since $v_\eps \in C([0,T]; D(A_2^m))$. In particular, $v_\eps$ is also 
  a variational solution of the equation
  \[
  v_\eps' + A v_\eps = f_\eps, \qquad v_\eps(0)=u_0^\eps.
  \]
  By construction we have that $v_\eps \to v$ in $C([0,T]; H)$;
  moreover, since $f_\eps\to f$ in $L^2(0,T; H)$ and $u_0^\eps\to u_0$
  in $H$, arguing as in the proof of Lemma~\ref{lm:sola} we have also
  that $v_\eps\to u$ in $C([0,T]; H)\cap L^2(0,T; V)$. Since mild and
  variational solutions are unique, we conclude that $u=v$.  We shall
  now extend this argument to the case where $u$ and $v$ are the
  unique variational and mild solutions to the equations
  \[
  u' + Au = F(\cdot,u), \qquad v'+A_2v = F(\cdot,v), \qquad u(0)=v(0)=u_0,
  \]
  respectively. Setting $f:=F(\cdot,v)$, the assumptions on $F$ imply
  that $f\in L^2(0,T; H)$, hence $v$ is a mild solution to
  $v'+A_2v=f$, $v(0)=u_0$. It then follows by the previous argument that
  $v$ is also the unique variational solution to $v' + Av =f$,
  $v(0)=u_0$.  Therefore
  \[
  u'+Au=F(\cdot,u), \qquad v'+Av = F(\cdot,v), \qquad u(0)=v(0)=u_0
  \]
  in the variational sense. Using the integration-by-parts formula,
  the Lipschitz continuity of $F$, and Gronwall's inequality, it is
  then a standard matter to show that $u=v$.
\end{proof}

The following estimates are crucial.
\begin{prop}   \label{prop:est}
  One has, for every $x,h,k\in H$ and $t>0$,
  \begin{align*}
    &\norm{Y^x_h}_{C([0,t];H)\cap L^2(0,t; V)} \lesssim \norm{h},\\
    &\norm{Z^x_{hk}}_{C([0,t];H)\cap L^2(0,t; V)} \lesssim_{\lambda,n}
      \norm{h}\norm{k},\\
    &\norm{Y^x_h}_{C([0,t];L^1(D))} \leq \norm{h}_{L^1(D)}.
  \end{align*}
  Regarding $A$ as an unbounded operator on $V'$, assume that there
  exists $\delta \in (0,1)$ and $\eta>0$ such that
  $H=\dom((\eta I+A)^\delta)$. Then
  \[
  \norm{Y^x_h(t)}_{H} \lesssim (1\vee t^{-\delta}) \norm{h}_{V'}.
  \]
\end{prop}
\begin{proof}
  Let $\Omega' \subseteq \Omega$ with $\P(\Omega')=1$ be such that
  \eqref{eq:rde} holds true for all $t \in [0,T]$ and all
  $\omega \in \Omega'$. Let $\omega \in \Omega'$ be fixed. Recalling
  that $A$ is coercive and that $\beta'_{\lambda n}$ is
  positive because $\beta_{\lambda n}$ is increasing, taking the
  scalar product with $Y_h(t)$ in \eqref{eq:rde} and integrating in time yields
  \[
  \frac12 \norm[\big]{Y^x_h(t)}^2 + C \int_0^t\norm{Y^x_h(s)}_V^2\,ds 
  \leq \frac12\norm{h}^2
  \]
  for all $t \in [0,T]$, and the first estimate is thus proved.
  The second estimate follows directly from
  Proposition~\ref{thm:d2}.
  Furthermore, denoting the Yosida approximation of the part of $A$ in
  $H$ by $A_\varepsilon$, let $Y_{h \eps}^x \in C([0,T];H)$ be the
  unique strong solution to the equation
  \[
  Y'_{h \eps}(t) + A_\eps Y_{h \eps}(t) 
  + \beta'_{\lambda n}(X_{\lambda n}(t)) Y_{h \eps}(t) = 0, 
  \qquad Y_{h \eps}(0) = h.
  \]
  Let $(\sigma_k)$ be a sequence of smooth increasing functions
  approximating pointwise the (maximal monotone) signum graph, and
  $\widehat{\sigma}_k$ be the primitive of $\sigma_k$ with
  $\widehat{\sigma}_k(0)=0$. Taking the scalar product of the previous
  equation with $\sigma_k(Y_{h \eps}^x)$ and integrating in time we
  get, for every $t>0$,
  \[
  \int_D \widehat{\sigma}_k(Y_{h \eps}^x(t)) 
  + \int_0^t\ip{A_\eps Y_{h \eps}^x(s)}{\sigma_k(Y_{h\eps}^x(s))}\,ds
  + \int_0^t\!\!\int_D \beta_{\lambda n}'(X_{\lambda n}(s))%
    \sigma_k(Y_{h \eps}^x(s))Y_{h\eps}^x(s)\,ds
  \leq \int_D \widehat{\sigma}_k(h).
  \]
  Since, as $k \to \infty$, $\sigma_k(Y_{h\eps}^x)$ converges a.e.  to
  a measurable function $w_\eps\in\operatorname{sgn}(Y_{h\eps}^x)$ and
  $\widehat{\sigma} \to |\cdot|$, letting $k\to\infty$ we get, for
  every $t \geq 0$,
  \[
  \norm{Y_{h \eps}^x(t)}_{L^1(D)} 
  + \int_0^t\ip{A_\eps Y_{h \eps}^x(s)}{w_\eps(s))}\,ds
  \leq \norm{h}_{L^1(D)}
  \qquad\forall t\in[0,T].
  \]
  Recalling that $A_2$ extends to an $m$-accretive operator on
  $L^1(D)$, the second term on the left-hand side is non-negative, and
  taking into account that $Y_{h\eps}^x\to Y_h^x$ in $C([0,T];H)$ as
  $\eps \to 0$, the third inequality follows.
  Finally, since $Y_h$ is the unique variational solution to
  \eqref{eq:rde}, by Lemma~\ref{lm:var-mild} we have that $Y_h$ is
  also mild solution to the same equation, i.e.
  \[
  Y^x_h(t) = S(t)h - \int_0^tS(t-s)\beta_{\lambda n}'(X^x(s))Y^x_h(s)\,ds 
  \qquad\forall t\in[0,T], \quad\P\text{-a.s.}
  \]
  Recall that $-A$ generates an analytic semigroup on $V'$ extending
  $S$, denoted by the same symbol. Since $H=\dom((\eta I+A)^\delta)$,
  we have $\norm{S(t)h} \lesssim t^{-\delta}\norm{h}_{V'}$ for every
  $t>0$.  By the contractivity of $S$ in $H$ we also have, for every
  $t>0$,
  \begin{align*}
  \norm{Y^x_h(t)} &\lesssim t^{-\delta}\norm{h}_{V'} +
  \norm{\beta_{\lambda n}'}_{\infty}\int_0^t\norm{Y^x_h(s)}\,ds 
  \end{align*}
  from which Gronwall's inequality implies
  \[
  \norm{Y^x_h(t)} \lesssim t^{-\delta}\norm{h}_{V'} 
  + \norm{\beta_{\lambda n}'}_{\infty}
  \int_0^t s^{-\delta}e^{\norm{\beta_{\lambda n}'}_{\infty}(t-s)}\norm{h}_{V'}\,ds.
  \]
  Therefore we have, for every $t \in [0,1]$,
  \[
  \norm{Y^x_h(t)} \lesssim t^{-\delta}\norm{h}_{V'} 
  + \norm{\beta_{\lambda n}'}_{\infty}
  e^{\norm{\beta_{\lambda n}'}_{\infty}} \norm{h}_{V'} \int_0^1s^{-\delta}\,ds
  = \left(t^{-\delta}+\frac{1^{1+\delta}}{1+\delta}\right) \norm{h}_{V'}
  \lesssim (1+t^{-\delta})\norm{h}_{V'}
  \]
  as well as, for every $t\geq1$,
  \[
  \norm{Y^x_h(t)} \leq \norm{Y^x_h(1)} \lesssim
  1^{-\delta}\norm{h}_{V'} = \norm{h}_{V'},
  \]
  which implies the last estimate.
\end{proof}

\begin{lemma}   \label{lem:v}

  Let $\alpha>0$ and $g\in C^1_b(V')\cap C^2_b(H)\cap C^1_b(L^1(D))$.
  For every $n\in\enne$ and $\lambda \in (0,1)$, the function
  $v_{\lambda n}: H \to \erre$ defined as
  \[
  v_{\lambda n}(x) := \E\int_0^{+\infty}e^{-\alpha t}g(X^x_{\lambda n}(t))\,dt
  \]
  belongs to $\dom(L_0)$ and solves \eqref{eq:Kreg}.  Moreover, there
  exists a positive constant $M$, independent of $\lambda$ and $n$,
  such that
  \begin{equation}   \label{eq:est_v}
  \norm{v_{\lambda n}}_{C^1_b(H)\cap C^1_b(L^1(D))} \leq M 
  \end{equation}
  for all $n \in \enne$ and $\lambda \in (0,1)$.
\end{lemma}
\begin{proof}
  Since $g\in C^1_b(H)$, for any $h\in H$ we have, by the
  first estimate of Proposition~\ref{prop:est},
  \begin{align*}
    D\left(g(X_{\lambda n}^x(t)\right)h
    &= Dg(X^x_{\lambda n}(t))DX^x_{\lambda n}(t)h=Dg(X_{\lambda n}^x(t))Y^x_h(t)\\
    &\leq\norm{Dg}_{C(H; H)}\norm{Y^x_h}_{C([0,T]; H)}%
      \leq \norm{Dg}_{C(H; H)}\norm{h},
  \end{align*}
  hence, by the dominated convergence theorem,
  $v_{\lambda n}\in C^1_b(H)$ and
  \begin{equation}
    \label{eq:dv}
    Dv_{\lambda n}(x)h = 
    \E\int_0^{+\infty}e^{-\alpha t}Dg(X_{\lambda n}^x(t))Y_h^x(t)\,dt.
  \end{equation}
  The uniform boundedness of $\norm{v_{\lambda n}}_{C^1_b(H)}$ in
  $\lambda$ and $n$ follows directly from these computations.
  Similarly, using the fact that $g\in C^2_b(H)$ and the second
  estimate of Proposition~\ref{prop:est}, we have, for every $k\in H$,
  \begin{align*}
  D(D(g(X_{\lambda n}^x(t))h)k
    &= D^2g(X_{\lambda n}^x(t))(Y_h^x(t),Y_k^x(t))%
      + Dg(X_{\lambda n}^x(t))Z_{hk}^x(t)\\
  &\leq\norm{D^2g}_{C(H;\cL_2(H;\erre))} \norm{Y_h^x}_{C([0,T];H)}
    \norm{Y_k^x}_{C([0,T]; H)}\\
  &\quad +\norm{Dg}_{C(H,H)}\norm{Z_{hk}^x}_{C([0,T]; H)}\\
  &\lesssim_{\lambda, n} \norm{g}_{C^2_b}\norm{h} \norm{k},
  \end{align*}
  hence, by the dominated convergence theorem, $v_{\lambda n} \in
  C^2_b(H)$ and
  \begin{equation}
    \label{eq:d2v}
  D^2v_{\lambda n}(x)(h,k) =
  \E\int_0^{+\infty}e^{-\alpha t}\left(D^2g(X_{\lambda n}^x(t))Y_h^x(t)Y_k^x(t)
  + Dg(X_{\lambda n}^x(t))Z_{hk}^x(t)\right)\,dt.
  \end{equation}
  Moreover, using the third estimate of Proposition~\ref{prop:est} and
  the fact that $g\in C^1_b(L^1(D))$, it follows by H\"older's
  inequality and \eqref{eq:dv} that
  \[
  Dv_{\lambda n}(x)h\leq \E\int_0^{+\infty}e^{-\alpha t}
  \norm{Dg}_{C(H; L^\infty(D))}\norm{Y_h^x(t)}_{L^1(D)}\,dt \leq
  \frac1\alpha\norm{Dg}_{C(H; L^\infty(D))}\norm{h}_{L^1(D)},
  \]
  which implies that $v_{\lambda n} \in C^1_b(L^1(D))$ and the
  estimate \eqref{eq:est_v}.
  Finally, by the last estimate of Proposition~\ref{prop:est} and the
  fact that $g\in C^1_b(V')$, we have
  \[
  Dv_{\lambda n}(x)h\leq \E\int_0^{+\infty}e^{-\alpha t}
  \norm{Dg}_{C(H; V)}\norm{Y_h^x(t)}_{V'}\,dt 
  \lesssim \norm{Dg}_{C(H;V)} \norm{h}_{V'} 
  \int_0^{+\infty} (1 \vee t^{-\delta}) e^{-\alpha t}\,dt.
  \]
  Since $t\mapsto (1\vee t^{-\delta})e^{-\alpha t}$ belongs to
  $L^1(0,+\infty)$, we have
  \[
  Dv_{\lambda n}(x)h \lesssim_{\lambda,n} \norm{h}_{V'},
  \]
  thus also $v_{\lambda n} \in C^1_b(V')$.

  Let us show now that $v_{\lambda n}$ solves \eqref{eq:Kreg}.
  Indeed, since $g \in C^2_b(H)\cap C^1_b(V')$, by It\^o's formula in
  the version of Proposition~\ref{prop:Ito_pard} we get
  \begin{align*}
    &g(X_{\lambda n}^x(t))
      + \int_0^t\ip{AX_{\lambda n}^x(s)}{Dg(X_{\lambda n}^x(s))}\,ds
    +\int_0^t\ip{\beta_{\lambda n}(X_{\lambda n}^x(s))}{Dg(X_{\lambda n}^x(s))}\,ds\\
    &\qquad =g(x) + \frac12\int_0^t\operatorname{Tr}[B^*(X_{\lambda
      n}^x(s))D^2g(X_{\lambda n}^x(s))B(X_{\lambda n}^x(s))]\,ds\\
    &\qquad\quad + \int_0^tDg(X_{\lambda n}^x(s))B(X_{\lambda n}^x(s))\,dW(s)
  \end{align*}
  for every $t>0$.  Thanks to the boundedness of
  $Dg$, taking expectations and using Fubini's theorem we deduce that,
  for every $\alpha>0$ and $x \in V$,
  \[
  e^{-\alpha t}\E g(X_{\lambda n}^x(t)) + \alpha\E \int_0^te^{-\alpha
    s}g(X_{\lambda n}^x(s))\,ds -\int_0^tP_s^{\lambda n}L_0^{\lambda
    n}g(x)\,ds = g(x).
  \]
  Since $g\in C_b(H)$, it is clear that, as $t\to +\infty$, the first
  and second term on the left-hand side converge to zero and $\alpha
  v_{\lambda n}(x)$, respectively, hence, by difference, we deduce
  that
  \[
  \int_0^tP_s^{\lambda n}L_0^{\lambda n}g(x) 
  \to \int_0^{+\infty}P_s^{\lambda n}L_0^{\lambda n}g(x)\,ds.
  \]
  Letting then $t\to+\infty$ we infer that
  \[
  \alpha v_{\lambda n}(x) - \int_0^{+\infty}e^{-\alpha t}P^{\lambda
    n}_tL_0^{\lambda n}g(x)\,dt=g(x),
  \]
  hence
  \[
  \alpha v_{\lambda n}(x) - L_0^{\lambda n} v_{\lambda n}(x) =g(x)
  \qquad \forall x\in V.
  \qedhere
  \]
\end{proof}

\begin{lemma}\label{lm:conv}
  One has
  \[
    \lim_{\lambda \to 0} \lim_{n \to \infty}
    \norm[\big]{L_0v_{\lambda n} - L_0^{\lambda n} v_{\lambda n}}_{L^1(H,\mu)}
    = 0.
  \]
\end{lemma}
\begin{proof}
  By definition of $L_0$ and $L_0^{\lambda n}$, the claim amounts to
  showing that
  \[
  \lim_{\lambda \to 0} \lim_{n \to \infty} \int_H \abs[\big]{%
  \ip[\big]{\beta_{\lambda n}(x)-\beta(x)}{Dv(x)}}\,\mu(dx) \to 0.
  \]
  Since $\beta_{\lambda n}$ is Lipschitz-continuous with Lipschitz
  constant bounded by $1/\lambda$ for every $n \in \enne$, we have
  \[
    \abs[\big]{%
      \ip[\big]{\beta_\lambda(x) - \beta_{\lambda n}(x)}{Dv(x)}}
    \lesssim \frac{1}{\lambda} \norm{x},
  \]
  so that, recalling that $\norm{\cdot} \in L^2(H,\mu)$ and
  $\beta_{\lambda n} \to \beta_\lambda$ pointwise as $n \to \infty$,
  the dominated convergence theorem yields
  \[
    \lim_{n \to \infty} \int_H \abs[\big]{%
      \ip[\big]{\beta_\lambda(x) - \beta_{\lambda n}(x)}{Dv(x)}}\,\mu(dx)
    = 0.
  \]
  Since $Dv_{\lambda n}(x)$ is bounded in $L^\infty(D)$ uniformly over
  $\lambda$, $n$ and $x$ by estimate \eqref{eq:est_v}, one has
  \[
    \abs[\big]{\bigl( \beta(x) - \beta_\lambda(x) \bigr) Dv(x)}
    \lesssim \abs[\big]{\beta(x) - \beta_\lambda(x)},
  \]
  hence
  \[
    \bigl( \beta(x) - \beta_\lambda(x) \bigr) Dv_{\lambda n}(x) \to 0
  \]
  in $L^0(D)$ as $\lambda \to 0$ for every $x \in V$.
  Recalling the definition of $\eta$ in \S\ref{sec:ass}, we deduce
  that $j^*(\eta|\beta(x)|)\in L^1(D)$.  Appealing to Young's
  inequality in the form 
  \[ 
  a\abs{b} \leq j(a) + j^*(\abs{b}) \qquad \forall a, b \in \erre,
  \]
  we have
  \[
    \eta\abs{\beta(x)} + \eta\abs{\beta_\lambda(x)} \leq 2j(1) +
    j^*(\eta|\beta(x)|) + j^*(\eta|\beta_\lambda(x)|)
  \]
  hence also, since $j^*$ is increasing on $\erre_+$ and
  $\abs{\beta_\lambda} \leq \abs{\beta}$,
  \[
    \abs[\big]{\bigl( \beta(x) - \beta_\lambda(x) \bigr) Dv_{\lambda n}(x)}
    \lesssim j(1) + j^*(\eta|\beta(x)|).
  \]
  which belongs to $L^1(D)$ for every $x \in J^*$. Therefore, by the
  dominated convergence theorem,
  \[
  \lim_{\lambda \to 0} \ip[\big]{\beta(x) - \beta_\lambda(x)}{Dv(x)} = 0
  \]
  for every $x \in H \cap J^*$. Using again the uniform boundedness in
  $L^\infty(D)$ of $v_{\lambda n}(x)$ we also have
  \[
  \abs[\big]{ \ip[\big]{\beta(x) - \beta_\lambda(x)}{Dv_{\lambda n}(x)} }
  \lesssim 1 + \int_D j^*(\delta|\beta(x)|),
  \]
  where the right-hand side belongs to $L^1(H,\mu)$ by
  Theorem~\ref{thm:supp}. A further application of the dominated
  convergence theorem thus yields
  \[
  \lim_{\lambda \to 0} \int_H 
  \abs[\big]{ \ip[\big]{\beta(x) - \beta_\lambda(x)}{Dv_{\lambda n}(x)} }\,\mu(dx)
  = 0.
  \qedhere
  \]
\end{proof}

We are now in the position to state and prove the main result of this
section, that gives a positive answer to the problem of
$L^1$-uniqueness for the Kolmogorov operator $L_0$. The question is
whether the extension to $L^1(H,\mu)$ of the transition semigroup $P$,
generated by the solution to the stochastic equation \eqref{eq:0}, is
the only strongly continuous semigroup on $L^1(H,\mu)$ whose
infinitesimal generator is an extension of the Kolmogorov operator
$L_0$.
Recall that, apart of the standing assumptions of
\S\ref{sec:ass}, we are also assuming that $\beta$ is a function, $B$
is non-random and does not depend on the unknown, $V$ is continuously
embedded in $L^4(D)$, and $H$ is the domain of a fractional power of
(a shift of) $A$, seen as the negative generator of an analytic
semigroup in $V'$.
\begin{thm}
  The generator $L$ of the extension to $L^1(H,\mu)$ of the transition
  semigroup $P$ is the closure of $L_0$ in $L^1(H,\mu)$.
\end{thm}
\begin{proof}
  Since the extension of the transition semigroup $P$ to $L^1(H,\mu)$
  is contractive, it follows by the Lumer-Phillips theorem that $L$ is
  $m$-accretive. As $L$ coincides with $L_0$ on $\dom(L_0)$, this
  implies that $L_0$ is accretive in $L^1(H,\mu)$, hence, in
  particular, closable. We are going to show that the image of
  $\alpha I+L_0$ is dense in $L^1(H,\mu)$ for all $\alpha>0$. Let $f \in
  L^1(H,\mu)$ and $\varepsilon>0$. Since $\dom(L_0)$ is dense in
  $L^1(H,\mu)$, there exists $g \in \dom(L_0)$ such that
  $\norm{f-g}_{L^1(H,\mu)}<\varepsilon/2$. Setting, for any $n \in
  \enne$ and $\lambda \in (0,1)$,
  \[
  v_{\lambda n}(x) := \int_0^\infty e^{-\alpha t} \E g(X^x_{\lambda n}(t))\,dt,
  \]
  if follows by Lemma~\ref{lem:v} that $v_{\lambda n} \in\dom(L_0)$
  and that
  \[
  \alpha v_{\lambda n}(x)+ L_0^{\lambda n}v_{\lambda n}(x) = g(x)
  \]
  for every $x\in V\cap J\cap J^*$, hence also
  \[
  \alpha v_{\lambda n}(x) + L_0v_{\lambda n}(x) - g(x) 
  = L_0v_{\lambda n}(x) - L_0^{\lambda n}v_{\lambda n}(x).
  \]
  Thanks to Lemma~\ref{lm:conv}, there exist $\lambda_0>0$ and $n_0
  \in \enne$ such that
  \[
    \norm[\big]{L_0v_{\lambda_0 n_0}%
      - L_0^{\lambda_0 n_0} v_{\lambda_0 n_0}}_{L^1(H,\mu)}
  < \varepsilon/2,
  \]
  hence, setting $\varphi:=v_{\lambda_0 n_0}$,
  \begin{align*}
  \norm[\big]{\alpha\varphi + L_0\varphi - f}_{L^1(H,\mu)} &\leq
  \norm[\big]{\alpha\varphi + L_0\varphi - g}_{L^1(H,\mu)} 
  + \norm[\big]{f-g}_{L^1(H,\mu)} < \varepsilon.
  \end{align*}
  As $\varepsilon>0$ was arbitrary, it follows that the image of
  $\alpha I + L_0$ is dense in $L^1(H,\mu)$. Since $L_0$ is closable,
  the Lumer-Phillips theorem implies that $-\overline{L_0}$, the
  closure of $-L_0$ in $L^1(H,\mu)$, generates a strongly continuous
  semigroup of contractions in $L^1(H,\mu)$. Recalling that $L$ is an
  extension of $L_0$, it follows again by the Lumer-Phillips theorem
  that $L=\overline{L_0}$ (see, for instance,
  \cite[Theorem~1.12]{Ebe}).
\end{proof}

%---------------------------------------------------

\ifbozza\newpage\else\fi
\bibliographystyle{amsplain}
\bibliography{ref}

\end{document}